\documentclass[12pt,reqno]{amsart}
\usepackage{amssymb}
\usepackage{amsfonts}
\usepackage{amsmath}
\usepackage{amsmath}
\usepackage{verbatim}
\usepackage{color}
\usepackage[shortlabels] {enumitem}
\usepackage[usenames,dvipsnames]{xcolor}

\usepackage[active]{srcltx}
\include{srctex}
\allowdisplaybreaks[4]
\usepackage[
bookmarks=true,         
bookmarksnumbered=true, 
colorlinks=true,
pdfstartview=FitV,
linkcolor=blue, citecolor=blue, urlcolor=blue]{hyperref}
\usepackage[abbrev,msc-links]{amsrefs}

\usepackage{fullpage}

\newtheorem{theorem}{Theorem}[section]

\newtheorem{corollary}[theorem]{Corollary}

\newtheorem{definition}[theorem]{Definition}

\newtheorem{proposition}[theorem]{Proposition}

\theoremstyle{remark}
\newtheorem{remark}[theorem]{Remark}

\let\Section=\section
\def\section{\setcounter{equation}{0}\Section}

\newcommand{\ep}{\varepsilon}

\newcommand{\si}{\sigma}

\newcommand{\cf}{\mathcal F}

\newcommand{\HH}{\mathfrak H}

\newcommand{\lc}{\left[}
\newcommand{\rc}{\right]}

\def\RR{\mathbb{R}}
\def\NN{\mathbb{N}}
\def\EE{\mathbb{E}}
\def\PP{\mathbb{P}}

\def\cee{\mathcal{E}}
\def\caa{\mathcal{A}}
\def\css{\mathcal{S}}
\def\cff{\mathcal{F}}

\def\lt{\left}
\def\rt{\right}
\begin{document}

\author{Jingyu Huang, Khoa L\^e \and David Nualart}
\address{Jingyu Huang: Department of Mathematics\\
The University of Utah\\
Salt Lake City, UT 84112}
\email{jhuang@math.utah.edu}
\address{Khoa L\^e: Departmentof Mathematics\\
University of Calgary\\
Calgary, Alberta, Canada\\
T2N 1N4}
\email{khoa.le@ucalgary.ca}
\address{David Nualart: {Department of Mathematics }\\
The University of Kansas \\
Lawrence, Kansas, 66045}
\email{nualart@ku.edu}
\title{Large time asymptotics for the parabolic Anderson model driven by space and time correlated noise}

\begin{abstract}
	 We consider the linear stochastic heat equation on $\RR^\ell$,  driven by a Gaussian noise which is colored in time and space.   The spatial covariance  satisfies general assumptions and includes examples such as the Riesz kernel in any dimension and the covariance of the fractional Brownian motion with  Hurst parameter $H\in (\frac 14, \frac 12]$ in dimension one. First we establish the existence of a unique mild solution and we derive a Feynman-Kac formula for its moments using a family of independent Brownian bridges and assuming a general integrability condition on the initial data.  In the second part of the paper we  compute Lyapunov exponents and lower and upper exponential growth indices in terms of a variational quantity. 
	 \end{abstract}
\maketitle

\section{Introduction}
{The purpose of this paper is to study the stochastic heat equation  }
	\begin{equation}\label{SHE}
		 \frac {\partial u}{\partial t} =\frac12\Delta u+u \diamond\frac{\partial^{\ell+1}W}{\partial t\partial x_1\dots\partial x_\ell} \, ,
	\end{equation}
	where $t\ge 0$, $x\in \RR^\ell$ $(\ell\ge1)$ and  $W$ is  a centered Gaussian field, which is correlated in both temporal and spatial variables. We assume that the noise $W$ is described by a centered  Gaussian family  $W=\{ W(\phi), \phi \in 
	\css(\RR_+\times\RR^\ell)\}$, with covariance
	\begin{equation}\label{eqn:fouriercov}
		\EE [W(\phi)W(\psi)]=\frac1{(2 \pi)^\ell} \int_0^\infty\int_0^\infty\int_{\RR^\ell} \cff\phi(s,\xi)\overline{\cff\psi(r,\xi)}\gamma_0(s-r)\mu (d \xi )dsdr,
	\end{equation}
	where $\gamma_0$ is  a nonnegative {and nonnegative definite} locally integrable function, $\mu$ is a {tempered measure} and $\cff$ denotes the Fourier transform in the spatial variables.  Throughout the paper, we denote by $|\cdot|$ the Euclidean norm in $\RR^\ell$ and by $x\cdot y$ the usual inner product between two vectors $x,y$ in $\RR^\ell$.  We are going to consider two types of spatial covariances:
	\medskip
	\noindent
	\begin{enumerate}[leftmargin=0cm,itemindent=1.4cm,label=(H.1)]
\item\label{H1} $\ell =1$, the spectral measure $\mu$ is absolutely continuous with respect to the Lebesgue measure on $\RR$ with density $f$, that is $\mu(d \xi)=f(\xi)d \xi$, and $f$ satisfies:  
\begin{itemize}
\item [(a)] For all $\xi,\eta$ in $\RR$ and for some constant $\kappa_0>0$,
\begin{equation}\label{eq:H11}
				f(\xi+\eta)\le  \kappa_0  (f(\xi)+f(\eta)).
\end{equation}
	\item	 [(b)] 
	\begin{equation}\label{eq:H12}
			\int_\RR\frac{f^2(\xi)}{1+|\xi|^2}d \xi<\infty\,.
		\end{equation}    
\end{itemize}		
\end{enumerate}	

	{To state the second type of covariance, we recall that the space of   Schwartz functions is
denoted by $\mathcal{S}(\RR^\ell)$.   The Fourier
transform of a function $u \in \mathcal{S}(\RR^\ell)$ is defined with the normalization
\[ \mathcal{F}u ( \xi)  = \int_{\mathbb{R}^\ell} e^{- i
   \xi \cdot x } u ( x) d x, \]
so that the inverse Fourier transform is given by $\mathcal{F}^{- 1} u ( \xi)
= ( 2 \pi)^{- \ell} \mathcal{F}u ( - \xi)$. }

	\noindent
	\begin{enumerate}[leftmargin=0cm,itemindent=1.4cm,label=(H.2)]
	\item\label{H2} The  inverse Fourier transform of $\mu$ is a nonnegative locally integrable function {(or generalized function)} denoted by $\gamma$
	\begin{equation}\label{eqn:gamspec}
		\gamma(x)=\frac1{(2 \pi)^\ell}\int_{\RR^\ell} e^{i \xi  \cdot x}\mu(d \xi)\,,
	\end{equation}
	 and $\mu$ satisfies Dalang's condition
	\begin{equation} \label{k5}
	\int_{ \RR^\ell}\frac {\mu(d\xi) }{1+ |\xi|^2} <\infty.
	\end{equation}	
\end{enumerate}
\medskip
For the case \ref{H2},  $\gamma$ is nonnegative definite  and \eqref{eqn:fouriercov} can be written as
	\begin{equation}\label{eqn:cov}
		\EE [W(\phi)W(\psi)]=\int_0^\infty\int_0^\infty\int_{\RR^{2\ell}} \phi(s,x)\psi(r,y)\gamma_0(s-r)\gamma(x-y)dxdydsdr\,.
	\end{equation}

	Examples of covariance functions satisfying {condition} \ref{H2} are the Riesz kernel $\gamma(x)=|x|^{-\eta}$, with $0<\eta <2\wedge \ell$, the space-time white noise in dimension one, where $\gamma =\delta_0$, and the multidimensional fractional Brownian motion, where $\gamma(x)= \prod_{i=1}^\ell H_i (2H_i-1) |x_i|^{2H_i-2}$, {assuming that} $\sum_{i=1}^\ell H_i >\ell-1$ and $H_i \in (\frac 12,1)$ for $i=1,\dots, \ell$.
 
 \medskip
In the case \ref{H1}, the inverse Fourier transform of $\mu$ is at best a distribution and the expression \eqref{eqn:gamspec} is only formal. The right-hand side of \eqref{eqn:cov}, however, makes sense by pairing between Schwartz functions and distributions. For our convenience, we will address $\gamma$ as  a \textit{generalized covariance function} if its Fourier transform {is} a (nonnegative) tempered measure. It also worths noting that by Jensen's inequality, \eqref{eq:H12} implies Dalang's condition \eqref{k5}.
 The basic example of a noise satisfying \ref{H1} is the {\it rough fractional noise}, where the spectral density is given by
  $f(\xi)=c_H |\xi|^{1-2H}$, with    $H \in (\frac 14, \frac 12]$ and $c_H= \Gamma(2H+1) \sin(\pi H)$.  In this case, the noise $W$ has the covariance of a fractional Brownian motion with Hurst parameter   $H \in (\frac 14, \frac 12]$ in the spatial variable.
  Condition (a) holds with $\kappa_0=1$ and condition (b) holds because of the restriction $H>\frac14$.

  \medskip
These types of spatial covariances were introduced in our paper \cite{HLN}, where the noise is white in time. In \cite{HLN} we proved the existence of a unique mild solution formulated using  It\^o-type  stochastic integrals, we derived Feynman-Kac formulas for the moments of the solution using a family of independent Brownian  bridges and we computed Lyapunov exponents and lower and upper exponential growth indices. The purpose of this paper is to carry out this program when the noise is not white in time. 
 While the general methodology of the current article is similar to \cite{HLN}, the case of colored temporal noises has some distinct features and needs a  different  treatment. In particular the existence and estimation of Lyapunov exponents offer new difficulties that  require techniques different from the white-time  case.   

 \medskip
After  a section on preliminaries, Section \ref{sec:ex-uniq-chaos} is devoted to show   (see Theorem \ref{thmSk1}) the existence of a unique mild solution to equation (\ref{SHE}),  when the stochastic integral is understood in the Skorohod sense, and the initial condition satisfies the general integrability condition (\ref{eq14}). We want to mention that the existence and uniqueness of a solution for the \ref{H2} type covariance in the case of a colored noise in time has been also obtained in the recent paper \cite{BC} by Balan and Chen.  Then,  in Section \ref{sec: FK mom} we establish  Feynman-Kac formulas for the moments of the solution in terms of independent Brownian motions or Brownian bridges (see Proposition \ref{prop1} and Corollary \ref{cor1}).  Section \ref{sec:Lyapunov-BB} is devoted to obtain Lyapunov exponents for exponential functionals of Brownian bridges,  assuming that  $\gamma_0(t)= |t|^{-\alpha_0}$ and  $\gamma(cx)=c^{-\alpha} \gamma(x)$ for all $c>0$, where $\alpha_0\in (0,1)$ and $\alpha \in(0,2)$.
The main result  of this section is Theorem  \ref{thm:expBB}, whose proof is inspired by Theorem 1.1 of  \cite{Chen}. 
While Chen's article \cite{Chen} deals with exponential functionals of Brownian motions, we deal with exponential functionals of Brownian bridges. 
Another difference is that we allow noises which satisfy condition \ref{H1}. In this case, the spatial covariance is generally a distribution and even if it is a function, it is not necessary non-negative and may switch signs. The former issue is solved by an appropriate approximation procedure. For the later issue, the compact folding argument in \cite{Chen} is no longer applicable here. Instead, we use a moment comparison between Brownian motions and Ornstein-Uhlenbeck processes, which is observed by Donsker and Varadhan in \cite{DV}.
We refer to 
\cite{CHSS} for related results on temporal asymptotics for the fractional parabolic Anderson model.

 \medskip
Finally, in Section \ref{sec: Exp Grow indices} we study the speed of propagation of intermittent peaks. The propagation of the farthest high peaks was  first considered by Conus and Khoshnevisan in \cite{CoKh} for a {one-dimensional} heat equation driven by a space-time white noise with compactly supported initial condition, where it is shown that there are intermittency fronts that move linearly with time as  $\lambda t$. More precisely, they defined the lower and upper exponential growth indices as follows:
\[
		\lambda_*(n)=\sup\left\{\lambda>0:\liminf_{t\to\infty}\frac 1t \sup_{|x|\ge \lambda t}\log\EE |u(t,x)|^n>0\right \}\,
\]
and
\[
		\lambda^*(n)=\inf\left\{\lambda>0:\limsup_{t\to\infty}\frac 1t \sup_{|x|\ge \lambda t}\log\EE  |u(t,x)|^n<0\right \}\,.
\]
Generalizing previous results by Chen and Dalang \cite{ChDa}, we proved in \cite{HLN} that, assuming that $u_0$ is nonnegative,
\begin{equation}  \label{bcn1}
 \sqrt{\frac{2\cee_n(\gamma)}{n}} \le 	 \lambda_*(n)\le  \lambda^*(n)  \le  \inf_{\beta:\int_{\RR^\ell} e^{\beta|y|}u_0(y)dy<\infty}\left( \frac {\beta}{2}+\frac{\cee_n(\gamma)}{n \beta}\right),	 
\end{equation}
where  $\cee_n(\gamma)$ is the $n$th Lyapunov exponent. In particular, If $u_0$ is nontrivial and supported on a compact set, then
		\begin{equation}\label{id:speed1}
			\lambda^*(n)=\lambda_*(n)=\sqrt{\frac{2\cee_n(\gamma)}{n}}\,.
		\end{equation}
In the reference \cite{HHN}, using the Feynman-Kac formula for the moments of the solution established in  \cite{HHNT},  the authors have  obtained lower and upper bounds for the exponential growth indices when the noise has a general covariance of the form $\EE[\dot{W}(t,x) \dot{W}(s,y)]=\gamma_0(t-s) \gamma(x-y)$, where $\gamma_0$ is locally integrable and the spatial covariance $\gamma$ satisfies \ref{H2}.
Here $\dot{W}(t,x)$ stands for $\frac {\partial ^{\ell+1}W}{\partial t \partial x_1 \cdots \partial x_\ell}$.  In this general situation, to obtain non-trivial limits the factor $t^{-1}$ and the set $\{|x| \ge \lambda t\}$ appearing in the definition of the exponential growth indices, need to be  changed. In the particular case $\gamma_0(t)=|t|^{-\alpha_0}$ and $\gamma(x)= |x|^{-\alpha}$ we need to replace  $t^{-1}$ by $t^{-a}$ and the set $\{|x| \ge \lambda t\}$ by $\{|x| \ge \lambda t^{\frac {a+1}2}\}$, where
$
a=\frac {4-\alpha -\alpha_0}{2-\alpha}.
$
In the present paper, we complete this analysis with the methodology developed in  \cite{HLN}, based on large deviations.

As a consequence of the  large deviation results obtained in Section \ref{sec:Lyapunov-BB}, in Section \ref{sec: Exp Grow indices}, under suitable scaling hypotheses on the covariance of the noise, we deduce  the following results on the exponential growth indices, that should be compared with (\ref{bcn1}) and (\ref{id:speed1}):  
\begin{itemize}
\item[(i)]  If the initial condition $u_0$ is a nonnegative function such that $ \int_{\RR^\ell} e^{\beta|y|^b}u_0(y)dy<\infty$, where $b= \frac{4- \alpha-2 \alpha_0}{3- \alpha- \alpha_0}$, then
\[
\lambda^*(n) \le    g_\beta^{-1} \left(    \left(  \frac {n-1} 2 \right)^{ \frac 2 {2-\alpha}}   \mathcal{E}(\alpha_0,\gamma)\right).
\]
where $g_\beta$ is an increasing function on $(0,\infty)$ defined by equation (\ref{gbeta}), and  $\mathcal{E}(\alpha_0,\gamma)$ is a variational quantity defined in  (\ref{egamma}). 
\item[(ii)]   Suppose $u_0$ is bounded below in a ball of radius $M$, and for some technical reasons  assume that the spatial covariance satisfies \ref{H2}.  Then,
	  \[
   \lambda _*(n) \ge    a^{\frac a2} (a+1) ^{-\frac {a+1}2}
  \sqrt{ 2  \left(  \frac {n-1} 2 \right)^{ \frac 2 {2-\alpha}}   \mathcal{E}(\alpha_0,\gamma) }\,,
	 \]\end{itemize}
where $a=\frac{4-\alpha-2\alpha_0}{2-\alpha}$. Moreover, as $\beta$ tends to infinity, the function $g_\beta(x)$ converges to $\sqrt{2x}$ and in the compact support case, the two bounds above differ only on the constant $a^{\frac a2} (a+1) ^{-\frac {a+1}2}$.

\section{Preliminaries} 
 \label{sec:preliminaries}

 Let $\mathcal{H}$  be the completion of
$\mathcal{S}(\RR_+\times\RR^\ell)$
endowed with the inner product
\begin{equation}\label{innprod1}
\langle \varphi , \psi \rangle_{\mathcal{H}}
=\frac  1 {(2\pi)^{\ell}}\int_{\RR_{+}^{2}\times\RR^{\ell}}  \cf \varphi(s,\xi) \overline{ \cf \psi(t,\xi)}\gamma_0(s-t) \mu(d\xi) \, dsdt.
\end{equation}
 The mapping $\varphi \rightarrow W(\varphi)$ defined on $\mathcal{S}(\RR_+\times\RR^\ell)$  can be extended to a linear isometry between
$\mathcal{H}$  and the Gaussian space
spanned by $W$. We will denote this isometry by
\begin{equation*}
W(\phi)=\int_0^{\infty}\int_{\RR^\ell}\phi(t,x)W(dt,dx)
\end{equation*}
for $\phi \in \mathcal{H}$.
If $\mu$ satisfies \ref{H2}, the righ-hand side of \eqref{innprod1} can be written in Cartesian coordinates as $\int_{\RR_+^2\times\RR^{2\ell}}\varphi(s,x)\psi(t,y)\gamma_0(s-t)\gamma(x-y)dxdydsdt$. Hence, a standard approximation (still assuming \ref{H2}) shows that $\mathcal{H}$  contains
the class of measurable functions $\phi$ on $\RR_+\times
\RR^\ell$  such that
\begin{equation}\label{abs1}
\int_{\RR^2_+  \times\RR^{2\ell}} |\phi(s,x)\phi(t,y)| \, \gamma_0(s-t)\gamma(x-y) \, dxdydsdt <
\infty\,.
\end{equation}

\subsection{Elements of Malliavin calculus}

We denote by $D$ the Malliavin derivative. That is, if $F$ is a smooth and cylindrical
random variable of the form
\begin{equation*}
F=f(W(\phi_1),\dots,W(\phi_n))\,,
\end{equation*}
with $\phi_i \in \mathcal{H}$, $f \in C^{\infty}_p (\RR^n)$ (that is, $f$ and all
its partial derivatives have polynomial growth), then $DF$ is the
$\mathcal{H}$-valued random variable defined by
\begin{equation*}
DF=\sum_{j=1}^n\frac{\partial f}{\partial
x_j}(W(\phi_1),\dots,W(\phi_n))\phi_j\,.
\end{equation*}
The operator $D$ is closable from $L^2(\Omega)$ into $L^2(\Omega;
\mathcal{H})$, and we define the Sobolev space $\mathbb{D}^{1,2}$ as
the closure of the space of smooth and cylindrical random variables
under the norm
\[
\|DF\|_{1,2}=\sqrt{\EE[F^2]+\EE[\|DF\|^2_{\mathcal{H}}]}\,.
\]
We denote by $\delta$ the adjoint of the derivative operator given
by the duality formula
\begin{equation}\label{dual}
\EE \lc \delta (u)F \rc =\EE \lc \langle DF,u
\rangle_{\mathcal{H}}\rc , 
\end{equation}
for any $F \in \mathbb{D}^{1,2}$ and any element $u \in L^2(\Omega;
\mathcal{H})$ in the domain of $\delta$. The operator $\delta$ is
also called the {\it Skorohod integral} because in the case of the
Brownian motion, it coincides with an extension of the It\^o
integral introduced by Skorohod. 

If $F\in \mathbb{D}^{1,2}$ and $h$ is an element of $%
\mathcal{H}$, then $Fh$ is Skorohod
integrable and, by definition, the
{\it Wick product } equals the Skorohod integral of $Fh$, that is,
\begin{equation}
\delta (Fh)=F\diamond W(h).  \label{Wick}
\end{equation}%
We refer to the book \cite{Nualart2} of Nualart
for a detailed account of the Malliavin calculus with respect to a
Gaussian process.

When handling the stochastic heat equation in the Skorohod sense we will make use of chaos expansions, which we briefly describe in the following.
For any integer $n\ge 0$ we denote by $\mathbf{H}_n$ the $n$th Wiener chaos of $W$. We observe that $\mathbf{H}_0$ is $\RR$ and for $n\ge 1$, $\mathbf {H}_n$ is the closed linear subspace of $L^2(\Omega)$ generated by the family of random variables $\{ H_n(W(h)), h \in \mathcal{H}, \|h\|_{\mathcal{H}}=1 \}$. Here $H_n$ is the $n$th Hermite polynomial.
For any $n\ge 1$, we denote by $\mathcal{H}^{\otimes n}$ (resp. $\mathcal{H}^{\odot n}$) the $n$th tensor product (resp. the $n$th  symmetric tensor product) of $\mathcal{H}$. Then, the mapping $I_n(h^{\otimes n})= H_n(W(h))$ can be extended to a linear isometry between    $\mathcal{H}^{\odot n}$, equipped with the modified norm $\sqrt{n!}\| \cdot\|_{\mathcal{H}^{\otimes n}}$, and $\mathbf{H}_n$.

Let us consider a random variable $F\in L^2(\Omega)$ which is measurable with respect to the $\sigma$-field  $\mathcal{F}^W$ generated by $W$. This random variable can be expressed (called the Wiener-chaos expansion of $F$) as
\begin{equation}\label{eq:chaos-dcp}
F= \EE \lc F\rc + \sum_{n=1} ^\infty I_n(f_n),
\end{equation}
where the series converges in $L^2(\Omega)$, and the elements $f_n \in \mathcal{H}^{\odot n}$, $n\ge 1$, are determined by $F$.

The Skorohod integral (or the divergence) of a random field $u$ can be
computed using the Wiener chaos expansion. More precisely,
suppose that $u=\{u(t,x),          (t,x) \in \RR_+ \times\RR^\ell\}$ is a random
field such that for each $(t,x)$, $u(t,x)$ is an
$\mathcal{F}^W$-measurable and square integrable random  variable.
Then, for each $(t,x)$, $u(t,x)$ has the Wiener chaos expansion of the form
\begin{equation}  \label{exp1}
u(t,x)= \EE\lc u(t,x)\rc + \sum_{n=1}^\infty I_n (f_n(\cdot,t,x)).
\end{equation}
Suppose additionally that the trajectories of $u$ belong to $\mathcal{H}$ and  $\EE[\| u\|^2_{\mathcal{H}}] <\infty$.
Then, we can interpret $u$ as a square  integrable
random function with values in $\mathcal{H}$ and the kernels $f_n$
in the expansion (\ref{exp1}) are functions in $\mathcal{H}
^{\otimes (n+1)}$ which are symmetric in the first $n$ time-space variables. In
this situation, $u$ belongs to the domain of the divergence (that
is, $u$ is Skorohod integrable with respect to $W$) if and only if
the following series converges in $L^2(\Omega)$
\begin{equation}\label{eq:delta-u-chaos}
\delta(u)= \int_0 ^\infty \int_{\RR^\ell}  u(t,x)   \delta W(t,x)= W(\EE[u]) + \sum_{n=1}^\infty I_{n+1} (\widetilde{f}_n),
\end{equation}
where $\widetilde{f}_n$ denotes the symmetrization of $f_n$ in all its $n+1$ time-space variables.

\subsection{Brownian bridges} 
\label{sub:brownian_bridges}
	Let $\{B(s), s\ge0\}$ be an $\ell$-dimensional Brownian motion starting at 0. For every fixed time $t>0$ and $x,y\in\RR^\ell$, the process
	\begin{equation*}
		\left \{\widetilde B(s)=x+B(s)-\frac st(B(t)+x-y), 0\le s\le t\right\}
	\end{equation*}
	is an $\ell$-dimensional Brownian bridge from $x$ to $y$, i.e. $\widetilde B(0)=x$ and $\widetilde B(t)=y$. Away from the terminal time $t$, the law of Brownian bridge admits a density with respect to Brownian motion. Indeed, it is shown in \cite[Lemma 3.1] {Nakao} that for every bounded measurable function $F$,
	\begin{multline}\label{id:density}
		\EE \lt[F(\{\widetilde B(s),0\le s\le \lambda t\})\rt]
		\\=(1- \lambda)^{-\frac\ell2} \EE \lt[\exp\left\{-\frac{|y-x-B(\lambda t)|^2}{2t(1- \lambda)}+\frac{|y-x|^2}{2t} \right\}F(\{ B(s),0\le s\le \lambda t\})\rt].
	\end{multline}
	Throughout the paper, we denote by $\{B_{0,t}(s),0\le s\le t\}$ an $\ell$-dimensional Brownian bridge which starts and ends at the origin. A Brownian bridge from $x$ to $y$ can be expressed as 
	\[\lt\{B_{0,t}(s)+\frac st y+(1-\frac st)x,0\le s\le t\rt\}\,.\]

\section{Existence and uniqueness of a solution via chaos expansions}
\label{sec:ex-uniq-chaos}
We  denote by $p_{t}(x)$ the
$\ell$-dimensional heat kernel $p_{t}(x)=(2\pi
t)^{-\ell/2}e^{-|x|^2/2t} $, for any $t > 0$, $x \in \RR^\ell$.
For each $t\ge 0$ let $\mathcal{F}_t$ be the $\sigma$-field generated by the random variables $W(\varphi)$, where
$\varphi$ has support in $[0,t ]\times \RR^\ell$. We say that a random field $u=\{u({t,x}), (t,x) \in \RR_+\times \RR^\ell\}$ is adapted if for each $(t,x)$ the random variable $u_{t,x}$ is $\mathcal{F}_t$-measurable.

We assume that the initial condition $u_0$ is a measurable function satisfying the condition
\begin{equation} \label{con:u0}
(p_t *|u_0|)(x) <\infty \mbox{ for all } t>0 \mbox{ and } x\in\RR^\ell\,,
\end{equation}
where $p_t *|u_0|$ denotes the convolution of the  heat kernel $p_t$ and the function $|u_0|$.   
This condition is equivalent to
\begin{equation}  \label{eq14}
\int_{\RR^\ell} e^{-\kappa|x|^2}  |u_0(x) |dx <\infty,
\end{equation}
for all $\kappa>0$.

We define the solution of equation  (\ref{SHE}) as follows.

\begin{definition}\label{def1}
An adapted   random field $u=\{u({t,x}), t \geq 0, x \in
\mathbb{R}^\ell\}$ such that $\EE  u^2({t,x}) < \infty$ for all $(t,x)$ is
a mild solution to equation (\ref{SHE}) with initial condition $u_0$ satisfying (\ref{eq14}), if for any $(t,x) \in [0,
\infty)\times \mathbb{R}^\ell$, the process $\{p_{t-s}(x-y)u({s,y}){\bf
1}_{[0,t)}(s) , s \ge 0, y \in \mathbb{R}^\ell\}$ is Skorohod
integrable, and the following equation holds
\begin{equation}\label{eq:sko-mild}
u(t,x)=p_t *u_0(x)+\int_0^t\int_{\mathbb{R}^\ell}p_{t-s}(x-y)u({s,y}) \, \delta
W_{s,y}.
\end{equation}
\end{definition}

Suppose now that $u=\{u({t,x}), t\geq 0, x \in \RR^\ell\}$ is a mild  solution to equation (\ref{eq:sko-mild}). Then according to (\ref{eq:chaos-dcp}), for any fixed $(t,x)$ the random variable $u({t,x})$ admits the following Wiener chaos expansion
\begin{equation}
u({t,x})=\sum_{n=0}^{\infty}I_n(f_n(\cdot,t,x))\,,
\end{equation}
where for each $(t,x)$, $f_n(\cdot,t,x)$ is a symmetric element in
$\mathcal{H}^{\otimes n}$.
Thanks to  (\ref{eq:delta-u-chaos}) and using an iteration procedure, one can then find an
explicit formula for the kernels $f_n$ for $n \geq 1$
\begin{eqnarray*}
f_n(s_1,y_1,\dots,s_n,y_n,t,x)=\frac{1}{n!}p_{t-s_{\si(n)}}(x-y_{\si(n)})\cdots p_{s_{\si(2)}-s_{\si(1)}}(y_{\si(2)}-y_{\si(1)})
p_{s_{\si(1)}}*u_0(y_{\si(1)})\,,
\end{eqnarray*}
where $\si$ denotes the permutation of $\{1,2,\dots,n\}$ such that $0<s_{\si(1)}<\cdots<s_{\si(n)}<t$
(see, for instance,  equation (4.4) in \cite{HN}, where this formula is established in the case of a noise which is white in space).
Then, to show the existence and uniqueness of the solution it suffices to show that for all $(t,x)$ we have
\begin{equation}\label{chaos}
\sum_{n=0}^{\infty}n!\|f_n(\cdot,t,x)\|^2_{\mathcal{H}^{\otimes n}}< \infty\,.
\end{equation}


\begin{theorem}\label{thmSk1}
Suppose that the spatial covariance satisfies \ref{H1} or \ref{H2}. 
Then relation (\ref{chaos}) holds
for each {$(t,x) \in (0,\infty)\times \RR^{\ell}$}. Consequently, equation (\ref{SHE}) admits a
unique mild solution in the sense of Definition \ref{def1}.
\end{theorem}

\begin{proof}
Notice that the kernel $f_n$ can be written as
\[
 f_n(s,y,t,x)= \int_{\RR^\ell} g_n(s,y,t,z) u_0(z) dz, 
\] 
where $s=(s_1, \dots, s_n)$, $y=(y_1,\dots, y_n)$ and
\begin{equation}\label{eq:def-gn}
 g_n(s,y,t,z)=\frac{1}{n!}p_{t-s_{\sigma(n)}}(x-y_{\sigma(n)})\cdots
p_{s_{\sigma(2)}-s_{\sigma(1)}}(y_{\sigma(2)}-y_{\sigma(1)})p_{s_{\sigma(1)}}(y_{\sigma(1)}-z)\,. 
\end{equation}
Then
\begin{eqnarray*}
n! \|f_n(\cdot,t,x)\|^2_{{\mathcal{H}}^{\otimes n}}
&=&\frac {n!}  {(2\pi)^{n\ell}} \int_{[0,t]^{2n}} \int_{(\RR^{\ell})^n}  \Phi(s,\xi) \overline{\Phi(r,\xi)}  \mu(d\xi)  \prod_{j=1}^n \gamma_0(s_j-r_j) dsdr \\
&\le& \frac {n!}  {(2\pi)^{n\ell}} \int_{[0,t]^{2n}}  \left( \int_{(\RR^{\ell})^n}  |\Phi(s,\xi)|^2 \mu(d\xi) \right)^{\frac 12} \\
&& \times 
\left( \int_{\RR^{n\ell}}  |\Phi(r,\xi)| ^2 \mu(d\xi) \right)^{\frac 12}  \prod_{j=1}^n \gamma_0(s_j-r_j) dsdr,
\end{eqnarray*}
where $\xi=(\xi^1, \dots, \xi^n)$, $\mu(d\xi)= \prod_{i=1}^n \mu(d\xi^i)$,
\[
 \Phi(s,\xi)=\int_{\RR^{\ell}} \mathcal{F}g_n(s,\cdot,t,z)(\xi) u_0(z)dz, 
\]
$ds=ds_1\cdots ds_n$ and $dr=dr_1\cdots dr _n$.
Using the inequality $ab \le \frac 12 (a^2+b^2)$ and the fact that $\gamma_0$ is locally integrable, we obtain
\[
n! \|f_n(\cdot,t,x)\|^2_{{\mathcal{H}}^{\otimes n}}
\le C^n n! \int_{[0,t]^{n}}   \int_{\RR^{n\ell}}  |\Phi(s,\xi)|^2 \mu(d\xi)  ds.
\]
By symmetry, this leads to
 \begin{equation} \label{qu1}
n! \|f_n(\cdot,t,x)\|^2_{{\mathcal{H}}^{\otimes n}}
\le C^n ( n!)^2 \int_{[0,t]_<^{n}}   \int_{\RR^{n\ell}}  |\Phi(s,\xi)|^2 \mu(d\xi)  ds,
\end{equation}
where for each $n\ge2$, we denote 
	\begin{equation}  \label{k8}
 		[0,t]^n_< := \{(t_1,\dots,t_n): 0 < t_1 < \cdots < t_n < t\}.
 	\end{equation}

 Fix $0<s_2<s_2< \cdots < s_n <t$. Notice that $(y,z) \mapsto n! g_n(s,y,t,z)$ is the joint density of the random vector
 $(B_{s_1}, B_{s_2}, \dots, B_{s_n}, B_t)$ at the point $(y_1-z, y_2-z, \dots, y_n-z,x-z)$ where $B=\{B_t, t\ge 0\}$ is an $\ell$-dimensional Brownian motion. Therefore,
 $n! g_n(s,\cdot,t,z) /p_t(x-z)$ is the conditional density of $(B_{s_1}, B_{s_2}, \dots, B_{s_n})$ given $B_t=x-z$, which coincides with the law of
 the random vector 
 \[
 Z= \left( B_{s_1} -\frac {s_1} t B_t+ \frac {s_1}t (x-z), \dots, B_{s_n} - \frac {s_n}t B_t+ \frac {s_n} t (x-z)\right).
 \]
 The characteristic function of this vector is given by
 \[
 \EE[e^{  i \xi \cdot Z }]= \exp\left(-\frac 12 \EE \left [ \left|\sum_{j=1} ^n \xi^j \cdot B_{0,t}(s_j) \right|^2 \right] \right) e^{i\frac {s_1 + \cdots +s_n} t  (x-z) \cdot  \xi},
 \] 
 where we recall that $\{B_{0,t}(s),  s\in [0,t]\}$ denotes an $\ell$-dimensional Brownian bridge from zero to zero.
 This implies
 \begin{equation*}
 |\Phi(s,\xi)| \le \frac 1{n!}  |p_t *u_0(x)|  \exp \left( -\frac 12 {\rm Var} \left( \sum_{j=1}^n \xi^j \cdot B_{0,t}(s_j) \right) \right).
 \end{equation*}
 Substituting the previous estimate into (\ref{qu1}) yields
  \begin{equation} \label{qu2}
n! \|f_n(\cdot,t,x)\|^2_{{\mathcal{H}}^{\otimes n}}
\le C^n   |p_t* u_0(x)|^2  \int_{[0,t]_<^{n}}   \int_{\RR^{n\ell}}   \exp \left( -{\rm Var} \left( \sum_{j=1}^n \xi^j \cdot B_{0,t}(s_j) \right) \right) \mu(d\xi)  ds.
\end{equation}
Finally, from Lemmas 9.1 and   9.4 of  \cite{HLN} we conclude that  (\ref{chaos}) holds. 
\end{proof}

\section{Feynman-Kac formulas for the moments of the solution}
\label{sec: FK mom}
For any $\varepsilon>0$, we define
$\gamma_\varepsilon$ by
	\begin{equation}\label{eqn:ge}
		\gamma_ \varepsilon(x)=\frac1{(2 \pi)^\ell}\int_{\RR^\ell} e^{- \varepsilon|\xi|^2} e^{i \xi \cdot x}\mu(d \xi)\,.
	\end{equation}
Notice that for each $\varepsilon>0$, the spectral measure of $\gamma_ \varepsilon$ is $\mu_\ep(d\xi):= e^{-\varepsilon |\xi|^2}\mu(d \xi)$, which has finite total mass because  $\mu$ is a tempered measure. Thus, $\gamma_ \varepsilon$ is a bounded positive definite function. 
The next proposition is the key ingredient in the proof of the Feynman-Kac formula for the moments  of the solution to equation (\ref{SHE}) using Brownian  bridges.  
 
	\begin{proposition}\label{prop:expbridge}
		Suppose that the spatial covariance satisfies \ref{H1} or \ref{H2}. Let $\kappa$ be a real number. 
		Let $\{B^j_{0,t}(s), s\in [0,t]\}$, $j=1\dots, n$, be independent $\ell$-dimensional Brownian  bridges from $0$ to $0$. 
		Then for each $\varepsilon>0$, the function $F_\varepsilon: (\RR^\ell)^n \rightarrow \RR$ given by
		\begin{equation*}
			F_ \varepsilon(x^1,\dots,x^n)=\EE\exp\left\{\kappa \sum_{1\le j<k\le n} \int_0^t \int_0^t  \gamma_0(s-r) \gamma_ \varepsilon(B_{0,t}^j(s)- B^k_{0,t}(r)+x^j-x^k)dr ds\right\}
		\end{equation*}
		is well-defined and continuous. Moreover, as $\varepsilon\downarrow0$, $F_ \varepsilon$ converges uniformly to a limit function denoted by 
		\begin{equation}\label{eqn:Fnx}
			\EE\exp\left\{\kappa \sum_{1\le j<k\le n} \int_0^t \int_0^t    \gamma_0(s-r) \gamma( B_{0,t}^j(s)- B^k_{0,t}(r)+x^j-x^k)drds\right\}\,.
		\end{equation}
	\end{proposition} 
	
	\begin{remark}
		Actually, for each $1 \le j<k \le n$, the integral
	\[
	\int_0^t 	\int_0^t  \gamma_0(s-r)  \gamma_{\varepsilon}( B_{0,t}^j(s)- B^k_{0,t}(r)+x^j-x^k)drds
	\]
	converges in $L^p(\Omega)$ as $\varepsilon$ tends to zero, for each $p\ge 1$, and we can also denote the limit as
	\[
	 \frac 1 {(2\pi)^{\ell}} \int_0^t\int_0^t\int_{\RR^\ell}  \gamma_0(s-r) e^{i \xi\cdot(B_{0,t}^j(s)- B^k_{0,t}(r)+x^j-x^k)}\mu(d \xi)drds.
	\]
	\end{remark}

	\begin{proof} 
		 We claim that  for every $\kappa\in\RR$
		\begin{equation}  \label{k9}
			\sup_{\varepsilon>0} \EE\exp\left\{\kappa \sum_{1\le j<k\le n} \int_0^t \int_0^t  \gamma_0(s-r) \gamma_ \varepsilon(B_{0,t}^j(s)- B^k_{0,t}(r))drds\right\}<\infty\,.
		\end{equation}
		By H\"older inequality, it suffices to show the previous inequality for $n=2$. For every    $d \in\NN$, we have 
	\begin{eqnarray*}
	&&\EE\left[   \int_0^t \int_0^t  \gamma_ \varepsilon(B_{0,t}^1(s)- B_{0,t}^2(r))\gamma_0(s-r) drds \right]^d \\  
	&&  \quad =\EE \left [  \frac 1 {(2\pi)^{\ell}}\int_0^t  \int_0^t \int_{\RR^\ell} e^{i \xi \cdot(B_{0,t}^1(s)- B_{0,t}^2(r))} \gamma_0(s-r) \mu_{\varepsilon}(d\xi) drds\right]^d\\ 
	&& \quad = \frac 1 {(2\pi)^{\ell d}}\int_{[0,t]^{2d}} \int_{(\RR^\ell)^d} \EE \exp\left\{i \sum_{k=1}^d \xi^k \cdot\left( B_{0,t}^1(s_k)- B_{0,t}^2(r_k) \right) \right\} \prod_{k=1}^d \gamma_0(s_k-r_k) \mu_\varepsilon(d\xi) drds\,,
	\end{eqnarray*}
	where we use the notation
	$\mu_\varepsilon(d\xi)  = \prod_{k=1}^d e^{-\varepsilon |\xi^k|^2}\mu(d\xi^k)$ and  $ds = ds_1 \cdots ds_d$.
	Using the  independence of $B^1$ and $B^2$, Cauchy-Schwarz inequality, the inequality $ab \le \frac 12 (a^2+b^2)$ and the fact that $\gamma_0$ is locally integrable, we obtain
		\begin{eqnarray}  \notag
	&&\EE\left[   \int_0^t \int_0^t  \gamma_ \varepsilon(B_{0,t}^1(s)- B_{0,t}^2(r))\gamma_0(s-r) drds \right]^d \\  
	&& \quad \le C^d \int_{[0,t]^d} \int_{(\RR^\ell)^d} \left| \EE \exp\left\{i \sum_{k=1}^d \xi^k \cdot B^1_{0,t}(s_k)  \right\} \right|^2 \mu_\varepsilon(d\xi) ds\\ \notag
	&& \quad  \le  C^d d!\int_{[0,t]^d_<} \int_{(\RR^\ell)^d} \exp\left\{-{\rm Var} \left(  \sum_{k=1}^d \xi^k \cdot B^1_{0,t}(s_k)  \right) \right\}\mu_\varepsilon(d\xi) ds,\\ \notag
	\label{eqn:d.int.form}
	\end{eqnarray}
	where   $[0,t]^d_<$ is defined in (\ref{k8}). Then (\ref{k9}) follows from the  Taylor expansion of $e^x$ and Lemmas 9.1 and 9.4 in \cite{HLN}. Finally, the proof of the uniform convergence of  $F_ \varepsilon$ as $\varepsilon$ tends to zero can be done by the same arguments as  in the proof of Proposition
4.2 in \cite{HLN}.  Notice that  Lemma 4.1 in \cite{HLN} has to be replaced by the inequality
\begin{multline}\label{est:GGy}
	\EE \exp\left\{\int_{[0,t]^2}  \sum_{1\le j<k\le n}  \kappa \gamma_\varepsilon( G_s^j- G_r^k+ y_{s,r}^{jk})drds\right\}
	\\\le \EE\exp\left\{\int_{[0,t]^2}\sum_{1\le j<k\le n}  |\kappa| \gamma_\varepsilon( G_s^j- G_r^k)dr ds\right\},
\end{multline}
where $\kappa \in \RR$, $G=(G^1,\dots,G^n)\in(\RR^\ell)^n$ is a centered Gaussian process indexed by $[0,t]$ and
		   $y=(y^{jk})_{1\le j<k\le n} :[0,t]^2\to(\RR^\ell)^{n(n-1)/2}$  is a measurable matrix-valued function.
	\end{proof}

	As an application, we have the following Feynman-Kac formula based on Brownian bridges.
	\begin{proposition}    \label{prop1}
	Suppose that the spatial covariance satisfies \ref{H1} or \ref{H2}.  Suppose that $\{B_{0,t}^j(s),s\in[0,t]\}$, $j=1,\dots, n$ are $\ell$-dimensional independent Brownian bridges from zero to zero.
		Then for every $x^1,\dots,x^n\in\RR^\ell$,
		\begin{align}\label{eqn:FKbridge}
			\EE\left[\prod_{j=1}^n u(t,x^j)\right]
			&=\int_{(\RR^\ell)^n}   \notag 
			\EE\exp\Big\{\int_{[0,t]^2} \sum_{1\le j<k\le n} \gamma\left(B_{0,t}^j(s)-B_{0,t}^k(r)+x^j-x^k+\frac st y^j- \frac rt y^k\right)\\
			 &\quad\times \gamma_0(s-r) drds \Big\}   \prod_{j=1}^n [u_0(x^j+y^j)p_t(y^j)]d y^1\cdots d y^n\,.
		\end{align}
	\end{proposition}
	\begin{proof} 
	For any $\ep>0$ we denote by $u_\ep(t,x)$ the solution to the stochastic heat equation
	\[
	\frac {\partial u_\ep }{\partial t }=\frac12  \Delta u_{\ep} +u_\ep \dot{W_\ep}\,,\quad u(0,\cdot)=u_0(\cdot)\,,
	\]
	where $\dot{W}_\ep$ is  a Gaussian centered  noise with covariance
	\[
	\EE [\dot{W}_\ep(s,y) \dot{W}_\ep(t,x)]= \gamma_0(s-t) \gamma_\ep(x-y).
	\]	
 From the results  of Hu, Huang, Nualart and Tindel  \cite{HHNT} we have the following Feynman-Kac formula for the moments of $u_\ep$ 
	 \begin{eqnarray} 
		\EE\left[ \prod_{j=1}^n u_\ep(t,x^j) \right] &=&\EE \Bigg( \prod_{j=1}^n u_0(B^j(t)+x^j)  \notag 
		  \exp\Big\{\sum_{1\le j<k\le n}\int_{[0,t]^2} \gamma_\ep(B^j(s)-B^k(r)+(x^j-x^k) )\\
		  &&\times \gamma_0(s-r) drds \Big\}\Bigg), \label{eqn:FKn}
	\end{eqnarray}
	where $\{B^j, j=1,\dots, n\}$ are independent $\ell$-dimensional standard Brownian motions.  We remark that in \cite{HHNT} it is required that $\gamma$ is a non-negative function, which is not necessarily true for $\gamma_\ep$. However,   $\gamma_\ep$ is bounded, and, in this case, it is not difficult to show that  (\ref{eqn:FKn})  still holds. Also, \cite{HHNT} assumes that $u_0$ is bounded, but it is not difficult to show that (\ref{eqn:FKn}) still holds assuming (\ref{eq14}). 
	
	For each $j=1,\dots,n$ and every fixed $t>0$, the Brownian motion $B^j$ admits the following decomposition
		\begin{equation}
			B^j(s)=B_{0,t}^j(s)+\frac st B^j(t),
		\end{equation}
		where $\{B_{0,t}^j(s),s\in[0,t]\}$, $j=1,\dots,n$, are Brownian bridges on $\mathbb{R}^\ell$ independent from $\{B^j(t), 1\le  j \le  n\}$ and from each other. Thus, identity \eqref{eqn:FKn} can be written as
	\begin{eqnarray}\label{eq5}
			\EE\left[\prod_{j=1}^n u_\varepsilon(t,x^j)\right]
			&=&\int_{(\RR^\ell)^n}   \notag 
			\EE\exp\Big\{\int_{[0,t]^2} \sum_{1\le j<k\le n} \gamma_\varepsilon \left(B_{0,t}^j(s)-B_{0,t}^k(r)+x^j-x^k+\frac st y^j- \frac rt y^k\right)\\
			 &&\times \gamma_0(s-r) drds \Big\}   \prod_{j=1}^n [u_0(x^j+y^j)p_t(y^j)]d y^1\cdots d y^n\,.
		\end{eqnarray}

		From Proposition \ref{prop:expbridge} and the  dominated convergence theorem, the right-hand side of 
		(\ref{eq5}) converges to the right-hand side of (\ref{eqn:FKbridge}).  From the Wiener chaos expansion of the solution  and the computations in  the proof of Theorem \ref{thmSk1}, it follows easily that $u_\ep(t,x)$ converges in $L^2(\Omega)$ to $u(t,x)$. On the other hand, from (\ref{eq5}) it follows that the moments of all orders of $u_\ep(t,x)$ are uniformly bounded in $\ep$. As a consequence, the left-hand side of 
		(\ref{eq5}) converges to the left-hand side of (\ref{eqn:FKbridge}).  This completes the proof. 
	\end{proof}
 
 \begin{corollary}
   \label{cor1} Under the assumptions of Proposition \ref{prop1} we have, for any $x\in \RR^\ell$
		\begin{equation}\label{k34}
			\EE\left[ u(t,x)^n\right]=
			\EE \Bigg( \prod_{j=1}^n u_0(B^j(t)+x)  
		  \exp\Big\{\sum_{1\le j<k\le n}\int_{[0,t]^2} \gamma(B^j(s)-B^k(r)) \gamma_0(s-r) drds \Big\}\Bigg),
		\end{equation}
		where $B^j$, $j=1,\dots, n$, are independent $\ell$-dimensional Brownian motions.
 \end{corollary}

	\begin{remark}
	If the initial condition $u_0$ is nonnegative, one can show that $u(t,x) \ge 0$ a.s., for all $t\ge 0$ and $x\in \RR^\ell$. 
	This follows from the fact that $u_\ep(t,x)$ is nonnegative for any $\ep$, where $u_\ep$ is the random field introduced in the proof of   Proposition \ref{prop1}.
	\end{remark}

\section{Lyapunov exponents  of Brownian bridges} 
\label{sec:Lyapunov-BB}
The following variational formula occurs frequently in our considerations,
 \begin{equation}  \label{egamma}
	\mathcal{E}(\alpha_0,\gamma)=\sup_{g \in \mathcal{A}_\ell} \left\{\int_{[0,1]^2} \int_{\RR^{2\ell}} \frac{\gamma(x-y)}{|s-r|^{\alpha_0}} g^2(s,x) g^2(r,y) dx dy dr ds- \frac{1}{2} \int_0^1 \int_{\RR^\ell} |\nabla_x g(s,x)|^2 dx ds \right\},
	\end{equation}
where   $\mathcal{A}_\ell$ is the  class of functions defined by
\begin{equation}   \label{aell}
\mathcal{A}_{\ell} = \left\{ g: g(s,\cdot) \in W^{1,2}(\RR^\ell) \ \text{and}\ \int_{\RR^{\ell}} g^2(s,x)dx=1, \, {\rm for  \,\, all} \, 0 \leq s \leq 1 \right\}\,,
\end{equation}
{where $\alpha_0 \in [0,1)$ and $\gamma$ is a generalized covariance function.}

In general, if $\eta_0$ is a covariance function {(nonnegative and nonnegative definite locally integrable function)} on $\RR$ and $\eta$ is a generalized covariance function on $\RR^\ell$ with spectral measure $\nu$, we can define the variational  {quantity}
\begin{multline}\label{def.eeta}
	\cee(\eta_0,\eta)
	=\sup_{g \in \mathcal{A}_\ell} \left\{\int_{[0,1]^2} \int_{\RR^{2\ell}}\eta(x-y)\eta_0(s-r) g^2(s,x) g^2(r,y) dx dy dr ds
	\rt.\\\lt.- \frac{1}{2} \int_0^1 \int_{\RR^\ell} |\nabla_x g(s,x)|^2 dx ds \right\}\,.
\end{multline}
It is evident that $\cee(\alpha_0,\gamma)=\cee(|\cdot|^{-\alpha_0},\gamma)$.
The first integration in \eqref{def.eeta} is defined through Fourier transforms,
	\begin{multline}\label{fourier.g}
		\int_{[0,1]^2} \int_{\RR^{2\ell}}\eta(x-y)\eta_0(s-r) g^2(s,x) g^2(r,y) dx dy dr ds
		\\= \frac 1 {(2\pi)^{\ell}}\int_{[0,1]^2} \int_{\RR^\ell}  \mathcal{F} g^2(s,\cdot)(\xi) \overline{\mathcal{F}g^2(r,\cdot)}(\xi) \nu(d\xi) \eta_0(s-r) drds\,.
	\end{multline}
	A priori, $\cee(\eta_0,\eta)$ can be infinite. However, if $\eta_0$ belongs to $L^1([-1,1])$ and $\eta$ satisfies the Dalang's condition \eqref{k5} (as in all cases in the current article), then $\cee(\eta_0,\eta)$ is finite. 
	Indeed, applying Cauchy-Schwarz inequality, we have
	\begin{align*}
		&\int_{[0,1]^2} \int_{\RR^\ell}  \mathcal{F} g^2(s,\cdot)(\xi) \overline{\mathcal{F}g^2(r,\cdot)}(\xi) \nu(d\xi) \eta_0(s-r) drds
		\\&\le \int_0^1\int_0^1\lt[\int_{\RR^\ell}|\cff g^2(s,\cdot)(\xi)|^2 \nu(d \xi) \rt]^{\frac12}\lt[\int_{\RR^\ell}|\cff g^2(r,\cdot)(\xi)|^2 \nu(d \xi) \rt]^{\frac12}\eta_0(s-r)dsdr
		\\&\le \frac12\int_0^1\int_0^1 \int_{\RR^\ell}|\cff g^2(s,\cdot)(\xi)|^2 \nu(d \xi)\eta_0(s-r)dsdr
		\\&\quad+\frac12\int_0^1\int_0^1 \int_{\RR^\ell}|\cff g^2(r,\cdot)(\xi)|^2 \nu(d \xi)\eta_0(s-r)dsdr
		\\&\le \|\eta_0\|_{L^1([-1,1])}\int_0^1\int_{\RR^\ell}|\cff g^2(s,\cdot)(\xi)|^2 \nu(d \xi)ds\,.
	\end{align*}
	Moreover, for each $s\in [0,1]$,  $ | \mathcal{F} g^2(s,\cdot)(\xi)|$ is bounded by $1$ and by $\frac {2\sqrt\ell}{|\xi|} \| \nabla_x g(s,\cdot)\|_{L^2(\RR^\ell)}$.  In fact, integrating by parts, we have
	\begin{align*}
		| \mathcal{F} g^2(s,\cdot)(\xi)| 
		&\le \min_{1\le j \le \ell} \frac 1{|\xi_j|}  \int_{\RR^\ell} \left| \frac {\partial g^2(s,x)}{\partial x_j} \right| dx 
	 \\&{=}  { \min_{1\le j \le \ell} \frac 1{|\xi_j|}   \left \|  \frac {\partial g^2(s,\cdot)}{\partial x_j} \right\| _{L^1(\RR^\ell)} \le  \frac {2\sqrt{\ell}}{|\xi|} \| \nabla_x g(s,\cdot)\|_{L^2(\RR^\ell)}. }
	\end{align*}
	It follows that for every $R>0$, 
	\begin{align*}
		&\int_0^1\int_{\RR^\ell}|\cff g^2(s,\cdot)(\xi)|^2 \nu(d \xi)ds
		\\&=\int_0^1\int_{|\xi|\le R}|\cff g^2(s,\cdot)(\xi)|^2 \nu(d \xi)ds+\int_0^1\int_{|\xi|>R}|\cff g^2(s,\cdot)(\xi)|^2 \nu(d \xi)ds
		\\&\le \int_{|\xi|\le R}\nu(d \xi)+4\ell\int_{|\xi|>R}\frac{\nu(d \xi)}{|\xi|^2}\int_0^1\int_{\RR^\ell}|\nabla_x g(s,x)|^2dxds\,.
	\end{align*}
	Since $\nu$ satisfies Dalang's condition $\int_{\RR^\ell}\frac{\nu(d \xi)}{1+|\xi|^2}<\infty$, we can choose $R>0$ such that \[
	{\| \eta_0\|_{L^1([-1,1])} }
	(2 \pi)^{-\ell}4\ell\int_{|\xi|>R}\frac{\nu(d \xi)}{|\xi|^2}<\frac12.
	\] This implies that the right-hand side of \eqref{def.eeta} is at most ${\| \eta_0\|_{L^1([-1,1])} }(2 \pi)^{-\ell}\int_{|\xi|<R}\nu(d \xi)$, which is also an upper bound for $\cee(\eta_0,\eta)$.

To conclude our discussion on basic properties of $\cee(\eta_0,\eta)$, we describe a useful comparison principle. Suppose $\eta_0,\tilde \eta_0$ are covariance functions on $\RR$ and $\eta,\tilde\eta$ are generalized covariance functions on $\RR^\ell$ such that the spectral measures of $\eta_0, \eta$ are less than the spectral measures of $\tilde \eta_0, \tilde \eta$ respectively. In other words, $\eta_0\le\tilde \eta_0$ and $\eta\le \tilde \eta$ in quadratic sense. Then
\begin{equation}\label{cee.compare}
	\cee(\eta_0,\eta)\le \cee(\tilde\eta_0,\tilde \eta)\,.
\end{equation}
This is immediate from \eqref{def.eeta}. 

\medskip
In the remaining of the article, we consider the following  scaling condition on the noise: 
\begin{enumerate}[leftmargin=0cm,itemindent=1.4cm,label=(S)]
	\item\label{con:S} There exist $\alpha_0\in (0,1)$ and $\alpha\in(0,2)$ such that $\gamma_0(t)=|t|^{-\alpha_0}$ and  $\gamma(cx)=c^{-\alpha} \gamma(x)$ for all $t,c>0$ and $x\in \RR^\ell$.
\end{enumerate}
 Under the scaling assumption \ref{con:S}, it is easy to check that  for every $\theta>0$,
\begin{equation}\label{id:Escale}
 	\mathcal{E}(\alpha_0,\theta\gamma)=\theta^{\frac2{2- \alpha}}\mathcal{E}(\alpha_0,\gamma)\,.
\end{equation} 
\begin{proposition}\label{prop:gg}
	Let $K$ and $\psi$ be {symmetric} functions in $L^2(\RR^\ell)$ and $L^2(\RR)$ respectively. We assume in addition that $\psi$ is nonnegative and  $\psi'$ exists and belongs to $L^2(\RR)$.
	The functions $\eta_0=\psi*\psi$ and $\eta=K*K$ are bounded and nonnegative definite functions. Then for every $\theta>0$ and every integer $n\ge1$,
	\begin{equation}\label{tmp.54}
		\limsup_{t\to\infty}\frac1t\log \EE\exp\lt\{\frac{\theta}{(n-1)t}\sum_{1\le j\neq k\le n}\int_0^t\int_0^t \eta(B^j(s)-B^k(r))\eta_0(\frac{s-r}t) dsdr \rt\}\le n  \cee(\theta\eta_0,\eta)\,,
	\end{equation}
	where $\cee(\eta_0,\eta)$ is the variational {quantity} defined in \eqref{def.eeta}.
\end{proposition}

Before giving the proof, let us explain our contribution.
This result, together with Theorem \ref{thm:expBB} below, extends the result of Chen in \cite{Chen}*{Section 4}, where $\eta$ is assumed to be nonnegative. In the aforementioned paper, the author uses a compact folding argument. When $\eta$ switches signs, this argument no longer works. In particular, \cite{Chen}*{inequality (4.15)} fails. Here, we replace the compact folding argument by a moment comparison between Brownian motions and Ornstein-Uhlenbeck processes, which was observed earlier by Donsker and Varadhan in  \cite{DV} (see \eqref{tmp.eek} below). Unlike Brownian motions, the Ornstein-Uhlenbeck processes are ergodic. This makes the essential arguments of \cite{Chen} carry through. Lastly, although the occupation times of Ornstein-Uhlenbeck processes satisfy (strong) large deviation principles, it cannot be applied here due to the time-dependent structure (namely $\eta_0$).
\begin{proof}[Proof of Proposition \ref{prop:gg}]
	We first observe that
	\begin{align*}
		&\sum_{1\le j\neq k\le n}\int_0^t\int_0^t \eta(B^j(s)-B^k(r))\eta_0(\frac{s-r}{t})dsdr
		\\&=\int_{\RR^{\ell+1}}\lt[\sum_{j=1}^n\int_0^t \psi(u-\frac st)K(x-B^j(s))ds \rt]^2dudx
		\\&\quad-\sum_{j=1}^n\int_{\RR^{\ell+1}}\lt[\int_0^t \psi(u-\frac st)K(x-B^j(s))ds \rt]^2dudx
		\\&\le (n-1)\sum_{j=1}^n\int_{\RR^{\ell+1}}\lt[\int_0^t \psi(u-\frac st)K(x-B^j(s))ds \rt]^2dudx\,.
	\end{align*}
	In conjunction with the independence of the Brownian motions, we see that the left-hand side of \eqref{tmp.54} is at most
	\begin{equation*}
		\limsup_{t\to\infty}\frac nt\log\EE\exp\lt\{\frac \theta t\int_{\RR^{\ell+1}}\lt[\int_0^t \psi(u-\frac st)K(x-B(s))ds \rt]^2dudx \rt\}\,.
	\end{equation*}
	Hence, it suffices to show
	\begin{equation}\label{tmp.55}
		\limsup_{t\to\infty}\frac 1t\log\EE\exp\lt\{\frac \theta t\int_{\RR^{\ell+1}}\lt[\int_0^t \psi(u-\frac st)K(x-B(s))ds \rt]^2dudx \rt\}\le  \cee(\theta\eta_0,\eta)\,.
	\end{equation}
	The proof is now divided into several steps.

	\medskip\noindent \textit{Step 1.} For each $\kappa>0$, let $\PP_\kappa$ be the law of an Ornstein-Uhlenbeck process in $\RR^\ell$ starting from 0 with generator $\frac12 \Delta- \kappa x\cdot \nabla$. Let $\EE_\kappa$ denote the expectation with respect to $\PP_\kappa$. We will show that
	\begin{multline}\label{tmp.eek}
		\EE\exp\lt\{\frac \theta t\int_{\RR^{\ell+1}}\lt[\int_0^t \psi(u-\frac st)K(x-B(s))ds \rt]^2dudx \rt\}
		\\\le \EE_ \kappa\exp\lt\{\frac \theta t\int_{\RR^{\ell+1}}\lt[\int_0^t \psi(u-\frac st)K(x-B(s))ds \rt]^2dudx \rt\}\,.
	\end{multline}
	We note that
	\begin{equation*}
		\int_{\RR^{\ell+1}}\lt[\int_0^t \psi(u-\frac st)K(x-B(s))ds \rt]^2dudx=\int_0^t\int_0^t \eta(B(s)-B(r))\eta_0(s-r)dsdr\,.
	\end{equation*}
	Hence, it suffices to check that for each integer $d\ge1$
	\begin{equation*}
		\EE\lt[\int_0^t\int_0^t \eta(B(s)-B(r))\eta_0(s-r)dsdr\rt]^d\le \EE_{\kappa}\lt[\int_0^t\int_0^t \eta(B(s)-B(r))\eta_0(s-r)dsdr\rt]^d\,.
	\end{equation*}
	Since $\eta_0$ is nonnegative, this amounts to show
	\begin{align}\label{tmp.56}
		\EE\lt[\prod_{j=1}^d \eta(B(s_j)-B(r_j)) \rt]
		\le \EE_\kappa\lt[\prod_{j=1}^d \eta(B(s_j)-B(r_j)) \rt]
	\end{align}
	for arbitrary times $s_1,r_1,\dots,s_d,r_d$ in $[0,t]$. 
	By writing $\eta(z)=(2 \pi)^{-\ell}\int_{\RR^\ell}e^{i \xi\cdot z}|\cff K(\xi)|^2 d \xi$, we see that
	\begin{align*}
		\EE\lt[\prod_{j=1}^d \eta(B(s_j)-B(r_j)) \rt]
		&=(2 \pi)^{-\ell d}\int_{\RR^{\ell d}}\EE e^{i\sum_{j=1}^d \xi_j\cdot (B(s_j)-B(r_j))}\prod_{j=1}^d|\cff K(\xi_j)|^2 d \xi_j
		\\&=(2 \pi)^{-\ell d}\int_{\RR^{\ell d}}e^{-\frac12\EE[ (\sum_{j=1}^d \xi_j\cdot (B(s_j)-B(r_j)))^2]}\prod_{j=1}^d|\cff K(\xi_j)|^2 d \xi_j\,.
	\end{align*}
	Hence, \eqref{tmp.56} is evident provided that
	\[
		\EE\lt[ \lt(\sum_{j=1}^d \xi_j\cdot (B(s_j)-B(r_j))\rt)^2\rt]\ge \EE_\kappa\lt[ \lt(\sum_{j=1}^d \xi_j\cdot (B(s_j)-B(r_j))\rt)^2\rt]\,.
	\]
	An observation made by Donsker-Varadhan \cite[proof of Lemma 3.10]{DV} is that $\EE[B(s)\otimes B(r)]\ge \EE_\kappa[B(s)\otimes B(r)]$ in quadratic sense. This fact implies the above inequality.

	\medskip\noindent \textit{Step 2.} As a consequence, \eqref{tmp.55} is reduced to showing
	\begin{equation}\label{tmp.57}
		\limsup_{\kappa\downarrow0}\limsup_{t\to\infty}\frac 1t\log\EE_\kappa\exp\lt\{\frac \theta t\int_{\RR^{\ell+1}}\lt[\int_0^t \psi(u-\frac st)K(x-B(s))ds \rt]^2dudx \rt\}\le  \cee(\theta\eta_0,\eta)\,.
	\end{equation}
	For each $t>0$ and each path $B$, we denote
	\begin{equation*}
		Z_{t,B}(u,x) =\frac1t\int_0^t \psi(u-\frac st)K(x-B(s))ds
	\end{equation*}
	and observe that
	\begin{align*}
		\int_{\RR^{\ell+1}}|Z_{t,B}(u,x)|^2dudx
		&=\frac 1{t^2}\int_{\RR^{\ell+1}}\lt[\int_0^t \psi(u-\frac st)K(x-B(s))ds \rt]^2dudx
		\\&=\frac1{t^2}\int_0^t\int_0^t \eta(B(s)-B(r))\eta_0\lt(\frac{s-r}t\rt)dsdr\,.
	\end{align*}
	In particular $Z_{t,B}$ belongs to $L^2(\RR^{\ell+1})$ and
	\begin{equation}\label{tmp.58}
		\|Z_{t,B}\|^2_{L^2(\RR^{\ell+1})}\le \eta(0)\eta_0(0)\,.
	\end{equation}
	Let $N$ be a fixed positive number and denote $\Omega_{t,N}=\{B:\frac1t\int_0^t|B(s)|ds\le N \}$. The only advantage of $\PP_\kappa$ over $\PP$, for which we need, is the following inequality
	\begin{equation}\label{tmp.Omc}
		\limsup_{t\to\infty}\frac1t\log\PP_\kappa(\Omega_{t,N}^c)\le -N+\frac1{2 \kappa^2}+\frac\ell2 \kappa  \,.
	\end{equation}
	In fact, by Girsanov's theorem we have
	\begin{align}
		\frac{d\PP_\kappa}{d\PP}\bigg|_{[0,t]}
		&=\exp\lt\{-\kappa\int_0^t B(s)\cdot dB(s)-\frac12 \kappa^2\int_0^t|B(s)|^2ds \rt\}
		\nonumber\label{tmp.Girsanov}\\&=\exp\lt\{-\frac12\kappa|B(t)|^2+\frac\ell2 \kappa t-\frac12 \kappa^2\int_0^t|B(s)|^2ds \rt\}\,.
	\end{align}
	It follows that
	\begin{align*}
		\EE_\kappa\left\{\exp\int_0^t|B(s)|ds\right\}
		&\le\EE\exp\lt\{\int_0^t\lt(|B(s)|-\frac12 \kappa^2|B(s)|^2\rt)ds+\frac\ell2 \kappa t \rt\}
		\\&\le\exp\lt\{\frac1{2 \kappa^2}t+\frac\ell2 \kappa t \rt\}
	\end{align*}
	where the last inequality is a consequence of a Cauchy-Schwarz inequality. Hence, in conjunction with Chebyshev's inequality, we obtain 
	\begin{align*}
		\PP_\kappa\lt(\frac1t\int_0^t|B(s)|ds>N \rt)\le e^{-Nt}\EE_\kappa e^{\int_0^t|B(s)|ds}\le e^{-Nt+\frac1{2 \kappa^2}t+\frac\ell2 \kappa t }\,.
	\end{align*}
	The estimate \eqref{tmp.Omc} is directly derived from here.

	The set $M=\{Z_{t,B}\}_{B\in \Omega_{t,N},t>0}$ is then a subset of $L^2(\RR^{\ell+1})$. We will show that $M$ is relatively compact in $L^2(\RR^{\ell+1})$. Indeed, we verify that $\cff M=\{\cff Z_{t,B}\}_{B\in \Omega_{t,N},t>0}$ satisfies the Kolmogorov-Riesz's compactness criterion in $L^2(\RR^{\ell+1})$ (cf. \cite[Theorem 5] {HO}). More precisely, we check that
	\begin{gather}
		\label{tmp.59} \sup_{B\in \Omega_{t,N},t>0} \|Z_{t,B}\|_{L^2(\RR^{\ell+1})}<\infty\,,
	 	\\\label{tmp.510}  \lim_{\rho \to\infty}\sup_{B\in \Omega_{t,N},t>0}\int_{|(\eta,\xi)|>\rho}|\cff Z_{t,B}(\eta,\xi)|^2d \eta d \xi=0\,,
	 	\\\label{tmp.511}\lim_{\rho\downarrow 0}\sup_{|(\tau',\xi')|<\rho} \sup_{B\in \Omega_{t,N},t>0}\int_{\RR^{\ell+1}}|\cff Z_{t,B}(\tau+\tau',\xi+\xi')-\cff Z_{t,B}(\tau,\xi)|^2d \tau d \xi=0 \,.
	\end{gather}
{Notice that}	\eqref{tmp.59} is evident from \eqref{tmp.58}. We can easily compute the Fourier transform of $Z_{t,B}$ 
	\begin{align*}
		\cff Z_{t,B}(\tau,\xi)=\cff \psi(\tau)\cff K(\xi) \frac1t\int_0^t e^{-i \tau\frac st-i \xi\cdot B(s)}ds\,.
	\end{align*}
	Hence,  
	\begin{align*}
		\sup_{B\in \Omega_{t,N},t>0}\int_{|(\tau,\xi)|>\rho}|\cff Z_{t,B}(\tau,\xi)|^2d \tau d \xi
		\le \int_{|(\tau,\xi)|>\rho}|\cff \psi(\tau)|^2|\cff K(\xi)|^2 d \tau d \xi\,,
	\end{align*}
	which implies \eqref{tmp.510}. To show \eqref{tmp.511}, let us first fix $\varepsilon>0$ and choose a function $g$ in $C_c^\infty(\RR^{\ell+1})$ such that $\|\cff \psi\otimes\cff K-g\|_{L^2(\RR^{\ell+1})}< \varepsilon$. We denote $Y_{t,B}(\tau,\xi)=g(\tau,\xi)\frac1t\int_0^t e^{-i \tau\frac st-i \xi\cdot B(s)}ds$ and observe that for every path $B $ in $\Omega_{t,N}$ and $|(\tau',\xi')|<\rho$, we have
	\begin{align*}
		&|Y_{t,B}(\tau+\tau',\xi+\xi')- Y_{t,B}(\tau,\xi)|
		\\&\le |g(\tau+\tau',\xi+\xi')-g(\tau,\xi)|
		+|g(\tau,\xi)|\lt|\frac1t\int_0^te^{-i \tau\frac st-i \xi\cdot B(s)}(e^{-i \tau'\frac st-i \xi'\cdot B(s)}-1)ds \rt|
		\\&\le |g(\tau+\tau',\xi+\xi')-g(\tau,\xi)|+2|g(\tau,\xi)| \lt(|\tau'|+ |\xi'|\frac1t\int_0^t|B(s)|ds\rt)\,.
		\\&\le |g(\tau+\tau',\xi+\xi')-g(\tau,\xi)|+2\rho(N+1)|g(\tau,\xi)|  \,.
	\end{align*}
	It follows that 
	\begin{align*}
		&\lim_{\rho\downarrow0}\sup_{|(\tau',\xi')|<\rho} \sup_{B\in \Omega_{t,N},t>0}\int_{\RR^{\ell+1}}|\cff Z_{t,B}(\tau+\tau',\xi+\xi')-\cff Z_{t,B}(\tau,\xi)|^2 d \tau d \xi
		\\&\le 4 \varepsilon^2+\lim_{\rho\downarrow0}\sup_{|(\tau',\xi')|<\rho} \sup_{B\in \Omega_{t,N},t>0}\int_{\RR^{\ell+1}}| Y_{t,B}(\tau+\tau',\xi+\xi')- Y_{t,B}(\tau,\xi)|^2 d \tau d \xi
		\\&\le 4 \varepsilon^2+\lim_{\rho\downarrow0}\sup_{|(\tau',\xi')|<\rho}\int_{\RR^{\ell+1}}| g(\tau+\tau',\xi+\xi')- g(\tau,\xi)|^2 d \tau d \xi \,.
	\end{align*}
	Since $g$ is uniformly continuous, the last limit above vanishes. Hence,
	\begin{align*}
		\lim_{\rho\downarrow 0}\sup_{|(\tau',\xi')|<\rho} \sup_{B\in \Omega_{t,N},t>0}\int_{\RR^{\ell+1}}|\cff Z_{t,B}(\tau+\tau',\xi+\xi')-\cff Z_{t,B}(\tau,\xi)|^2d \tau d \xi\le 4 \varepsilon^2
	\end{align*}
	for every $\varepsilon>0$. This in turn implies \eqref{tmp.511}.

	\medskip\noindent \textit{Step 3.} Applying \eqref{tmp.58},
	\begin{align*}
	 	\EE_\kappa e^{t\theta \|Z_{t,B}\|^2_{L(\RR^{\ell+1})}}
	 	&\le \EE_\kappa \lt[{\bf 1}_{\Omega_{t,N}} e^{t \theta \|Z_{t,B}\|^2_{L^2(\RR^{\ell+1})}}\rt]+\PP_\kappa(\Omega_{t,N}^c)e^{t \theta \eta(0)\eta_0(0)}\,.
	\end{align*}
	Together with \eqref{tmp.Omc}, the previous estimate yields
	\begin{multline}\label{20.512}
		\limsup_{t\to\infty}\frac1t\log\EE_\kappa e^{t\theta \|Z_{t,B}\|^2_{L(\RR^{\ell+1})}}\le (\theta \eta(0)\eta_0(0)-N+\frac1{2 \kappa^2}+\frac\ell2 \kappa )\vee
		\\\limsup_{t\to\infty}\frac1t\log\EE_\kappa \lt[{\bf 1}_{\Omega_{t,N}} e^{t \theta \|Z_{t,B}\|^2_{L^2(\RR^{\ell+1})}}\rt]\,.
	\end{multline}
	To deal with the limit on the right-hand side above, we adopt an argument from \cite{Chen}. Let $\varepsilon$ be a fixed positive number and define
	\begin{equation*}
		\mathcal O_h=\{g\in L^2(\RR^{\ell+1}):\|g\|^2<-\|h\|^2+2\langle g,h\rangle+\varepsilon \}\,.
	\end{equation*}
	The collection $\{\mathcal O_h\}_{h\in M}$ forms an open cover of $M$ in $L^2(\RR^{\ell+1})$. Since $M$ is relatively compact, we can find deterministic functions $h_1,\dots,h_m\in M$ such that $M\subset\cup_{j=1}^m \mathcal O_{h_j}$. It follows that for every $t>0$ and $B\in \Omega_{t,N}$,
	\begin{equation*}
		\|Z_{t,B}\|_{L^2(\RR^{\ell+1})}^2<\max_{j=1,\dots,m}\lt(-\|h_j\|^2_{L^2(\RR^{\ell+1})}+2 \langle h_j,Z_{t,B}\rangle_{L^2(\RR^{\ell+1})}+\varepsilon \rt)
	\end{equation*}
	and hence,
	\begin{multline}\label{20.513}
		\limsup_{t\to\infty}\frac1t\log\EE_\kappa \lt[{\bf 1}_{\Omega_{t,N}} e^{t \theta \|Z_{t,B}\|^2_{L^2(\RR^{\ell+1})}}\rt]
		\\\le \max_{j=1,\dots,m}\lt(-\theta\|h_j\|^2_{L^2(\RR^{\ell+1})}+\varepsilon \theta+ \limsup_{t\to\infty}\frac1t\log\EE_\kappa e^{2t \theta\langle h_j,Z_{t,B}\rangle_{L^2(\RR^{\ell+1})}}\rt)\,.
	\end{multline}
	We note that
	\begin{align*}
		\langle h_j,Z_{t,B}\rangle_{L^2(\RR^{\ell+1})}
		&=\frac1t\int_{\RR^{\ell+1}}h_j(u,x)\lt[\int_0^t \psi(u-\frac st)K(x-B(s))ds\rt]dudx
		\\&=\frac1t\int_0^t\bar h_j(\frac st,B(s))ds\,,
	\end{align*}
	where
	\begin{equation*}
		\bar h_j(s,z)=\int_{\RR^{\ell+1}}h_j(u,x) \psi(u-s)K(x-z)dudx=h_j*(\psi\otimes K)(s,z) \,.
	\end{equation*}
	Since $\bar h_j$ is the convolution of $L^2$-functions, it is continuous and bounded. Moreover, since $\psi'$ belongs to $L^2(\RR)$, $\partial_s\bar h_j$ exists and $\|\partial_s\bar h_j\|_{L^\infty}\le\|h_j\|_{L^2}\|\partial_s \psi\otimes K\|_{L^2}$. In particular, $\bar h_j$ satisfies the hypothesis of \cite[Proposition 3.1] {CHSX}.
	We also note that from \eqref{tmp.Girsanov}, $\frac{d\PP_k}{d\PP}\big|_{[0,t]}\le e^{\frac\ell2 \kappa t}$.
	In conjunction with \cite[Proposition 3.1] {CHSX}, it follows that
	\begin{align*}
		&\limsup_{t\to\infty}\frac1t\log\EE_{\kappa}e^{2t \theta\langle h_j,Z_{t,B}\rangle_{L^2(\RR^{\ell+1})}}
		\\&\le \frac\ell2 \kappa+\limsup_{t\to\infty}\frac1t\log\EE \exp\lt\{2\theta\int_0^t\bar h_j(\frac st,B(s))ds \rt\}
		\\&\le \frac\ell2 \kappa+\sup_{g\in\caa_\ell}\lt\{2 \theta\int_0^1\int_{\RR^{\ell}}\bar h_j(s,x)g^2(s,x)dxds-\frac12\int_0^1\int_{\RR^\ell}|\nabla_xg(s,x)|^2dxds \rt\}
		\\&= \frac\ell2 \kappa+\sup_{g\in\caa_\ell}\lt\{2 \theta\langle h_j,(\psi\otimes K)* g^2\rangle_{L^2(\RR^{\ell+1})} -\frac12\int_0^1\int_{\RR^\ell}|\nabla_xg(s,x)|^2dxds \rt\}\,,
	\end{align*}
	where each for $g\in\caa_\ell$ we conventionally set $g(s,x)=0$ if $s\notin[0,1]$. Gluing together our argument since \eqref{20.513}, we obtain
	\begin{align*}
		&\limsup_{t\to\infty}\frac1t\log\EE_\kappa\lt[{\bf 1}_{\Omega_{t,N}}e^{t \theta\|Z_{t,B}\|^2_{L^2(\RR^{\ell+1})}} \rt]
		\\&\le\max_{j=1,\dots,m}\sup_{g\in\caa_\ell}\lt\{-\theta\|h_j\|^2_{L^2(\RR^{\ell+1})}+2 \theta\langle h_j,(\psi\otimes K)*g^2\rangle_{L^2(\RR^{\ell+1})}-\frac12\int_0^1\int_{\RR^\ell}|\nabla_xg(s,x)|^2dxds \rt\}
		\\&\quad+ \varepsilon \theta+\frac\ell2 \kappa  	\,.
	\end{align*}  
	Applying the Cauchy-Schwarz inequality $-\|h_j\|^2_{L^2}+2 \langle h_j,(\psi\otimes K)*g^2\rangle_{L^2}\le\|(\psi\otimes K)*g^2\|^2_{L^2}$, we further have
	\begin{align*}
		&\limsup_{t\to\infty}\frac1t\log\EE_\kappa\lt[{\bf 1}_{\Omega_{t,N}}e^{t \theta\|Z_{t,B}\|^2_{L^2(\RR^{\ell+1})}} \rt]
		\\&\le \sup_{g\in\caa_\ell}\lt\{\theta\|(\psi\otimes K)*g^2\|^2_{L^2(\RR^{\ell+1})}-\frac12\int_0^1\int_{\RR^\ell}|\nabla_xg(s,x)|^2dxds \rt\}
		+\varepsilon \theta+\frac32 \kappa\ell  	\,.
	\end{align*}
	Together with \eqref{tmp.eek}, \eqref{20.512} and the fact that
	\begin{equation*}
		\|(\psi\otimes K)*g^2\|^2_{L^2(\RR^{\ell+1})}=\int_0^1\int_0^1\int_{\RR^\ell\times\RR^\ell}\eta(x-y)\eta_0(s-r)g^2(s,x)g^2(r,y)dxdydsdr
	\end{equation*}
	we see that the left-hand side of \eqref{tmp.55} is at most
	\begin{align*}
		\lt(\theta \eta(0)\eta_0(0)-N+\frac1{2 \kappa^2}+\frac\ell2 \kappa\rt)\vee \lt(\cee(\theta \eta_0,\eta)+\varepsilon \theta+\frac\ell2 \kappa\rt)\,.
	\end{align*}
	We now send $N\to\infty$, $\kappa\downarrow0$ and $\varepsilon\downarrow0$ to obtain \eqref{tmp.55} and complete the proof.
\end{proof}
The following result  provides an upper bound for the Lyapunov exponents of Brownian bridges.
\begin{theorem}\label{thm:expBB}
Suppose that  the covariance of the noise  satisfies condition \ref{H1} or \ref{H2} and also \ref{con:S}.  Assume that the spectral density $f(\xi)$ exists. 
Suppose that $\{B^j_{0,t} (s), s\in [0,t]\}$, $j=1\dots, n$, are independent $\ell$-dimensional Brownian bridges from zero to zero.
Then, 
	\begin{equation}\label{upper bd 2}  
		\limsup_{t \to \infty} t^{-a}\log \EE \exp \left \{ \sum_{0\le j< k\le n} \int_{[0,t]^2}\frac{\gamma(B_{0,t}^j(s)-B_{0,t}^k (r))}{|s-r|^{\alpha_0}} dr ds \right\} \le n \left(\frac{n-1}{2} \right)^{\frac{2}{2-\alpha}} \mathcal{E}(\alpha_0,\gamma)\,,
	\end{equation} 	
	where we recall that $a=\frac{4-\alpha - 2\alpha_0}{2-\alpha}$.
\end{theorem}

\begin{proof}
For suitable distributions $\eta_0,\eta$, we are going  to make use of the notation
	\[
Q_t(\eta_0,\eta):= \sum_{0\le j<k \le n}  \int_{[0,t]^2} \eta(B_{0,t}^j(s)-B_{0,t}^k (r))\eta_0(\frac{s-r}t)drds\,.
	\]
Let $\gamma_0(s)=|s|^{-\alpha_0}$ denote the temporal covariance. 
With these notation, the expectation in \eqref{upper bd 2} can be written as $\EE\exp\lt\{t^{-\alpha_0}Q_t(\gamma_0,\gamma)  \rt\}$.
We note that $\gamma_0=\psi*\psi$ where $\psi(s)=c(\alpha_0) |s|^{-\frac{1+\alpha_0}2} $ with some suitable constant $c(\alpha_0)$. For each $\delta>0$ we set $\psi_\delta=p_{\delta/2}* \psi$ and $\gamma_{0,\delta}=\psi_\delta*\psi_\delta$. To prove \eqref{upper bd 2}, the main ideas are first approximate the singular covariances $\gamma_0,\gamma$  by regular covariances $\gamma_{0,\delta},\gamma_\varepsilon$; then upper bound the exponential functional of Brownian bridges by that of Brownian motions. At the final stage, we will apply Proposition \ref{prop:gg}.
This procedure will be carried out in detail in several steps below. For the moment, let us put $t_n=(n-1)^{\frac2{2- \alpha}}t^a$ and observe that by making change of variables $s\to\frac t{t_n}s$, $r\to\frac t{t_n}r$ and using the scaling properties of $\gamma$ and of Brownian bridges (i.e. \ref{con:S} and $\{B_{0,t}(\lambda s),s\le t/ \lambda \}\stackrel{d}{=} \sqrt{\lambda} \{B_{0,\frac{t}{\lambda}}(s),s\le t/\lambda\} $ for any $\lambda>0$) we have
\begin{equation}\label{eq:scaling}
	\EE \exp \left\{t^{-\alpha_0} Q_t(\gamma_0, \gamma )\right\}= \EE \exp \left\{  \frac {1} {(n-1)t_n} Q_{t_n} (\gamma_0,\gamma) \right\}.
\end{equation}
Therefore, in conjunction with \eqref{id:Escale}, \eqref{upper bd 2} is equivalent to
\begin{equation}\label{tmp.eqt}
	\limsup_{t\to\infty}\frac1t\log \EE \exp \left\{  \frac {1} {(n-1)t} Q_{t} (\gamma_0,\gamma) \right\}\le n \cee(\frac12\gamma_0,\gamma)\,.
\end{equation}

	\medskip
	\noindent
	{\it Step 1.}  
Fix $\ep>0$. 	      For any $p,q>1$, $\frac 1p + \frac 2q =1$, applying H\"older inequality, we have
\begin{align}  \label{65}
		\log \EE e^{\frac{1}{(n-1)t} Q_t(\gamma_0, \gamma)}  
		&\le \frac1p\log \EE e^{ \frac p{(n-1)t}Q_t\left( \gamma_{0,\delta},\gamma_\ep \right)} 
		\\\notag&\quad+\frac1q\log \EE e^{\frac q{(n-1)t}Q_t \left( \gamma_0,\gamma -\gamma_\ep  \right)}
		+\frac1q\log \EE e^{ \frac q{(n-1)t}Q_t \left( \gamma_0- \gamma_{0,\delta},\gamma_\ep  \right)}\,.
\end{align}
We claim that
 \begin{equation} \label{ecua11}
	\limsup_{\varepsilon \downarrow 0} \limsup_{t\rightarrow \infty} \frac 1t \log \EE e^{\frac q{(n-1)t}Q_t \left( \gamma_0,\gamma -\gamma_\ep  \right)} \le 0
 \end{equation}
 and
 \begin{equation}\label{ecua22}
 	{\limsup_{\varepsilon \downarrow 0}}\limsup_{\delta\downarrow0}\limsup_{t\to\infty}\frac1t\log \EE e^{ \frac q{(n-1)t}Q_t \left( \gamma_0- \gamma_{0,\delta},\gamma_\ep  \right)}\le 0\,.
 \end{equation}
 Let us focus on \eqref{ecua11}. By H\"older's inequality, it suffices to show  that for any $\kappa \in \RR$
 \begin{equation} \label{ecua12}
 {\limsup_{\varepsilon \downarrow 0}}  \limsup_{t\rightarrow \infty}  \frac 1t	  \log \EE \exp \left \{  \kappa t^{\alpha_{0}-1}   \int_{[0,t]^2}\frac{(\gamma-\gamma_\varepsilon)(B_{0,t}^1(s)-B_{0,t}^2 (r))}{|s-r|^{\alpha_0}} dr ds \right\}  \le 0.
 \end{equation}
 For each integer $d\ge1$, we can write
 \begin{eqnarray*}
&& \EE   \left [\int_{[0,t]^2}\frac{(\gamma-\gamma_\varepsilon)(B_{0,t}^1(s)-B_{0,t}^2 (r))}{|s-r|^{\alpha_0}} dr ds \right]^d\\
&& \quad  = \frac 1 {(2\pi)^{\ell d}}\int_{[0,t]^{2d}  }  \int_{(\RR^\ell)^d} \EE e^{ i \sum_{j=1}^d  \xi^j \cdot (B^1_{0,t} (s_j)- B^2_{0,t}(r_j)}  \prod_{j=1}^d  |s_j-r_j|^{-\alpha_0}  (1- e^{-\varepsilon |\xi^j|^2})  \mu(d\xi) drds.
 \end{eqnarray*}
 Then, using Cauchy-Schwarz inequality and the inequality $ab \le \frac 12(a^2+b^2)$, we obtain
  \begin{eqnarray*}
&& \EE   \left [\kappa t^{\alpha_0-1} \int_{[0,t]^2}\frac{(\gamma-\gamma_\varepsilon)(B_{0,t}^1(s)-B_{0,t}^2 (r))}{|s-r|^{\alpha_0}} dr ds \right]^d
  \\
  && \quad \le C^d \int_{[0,t]^{d}  }  \int_{(\RR^\ell)^d}  \left| \EE e^{ i \sum_{j=1}^d  \xi^j \cdot B^1_{0,t} (s_j)} \right|^2   \prod_{j=1}^d  (1- e^{-\varepsilon |\xi^j|^2})  \mu(d\xi) ds\\
  && \quad  = \EE  \left [C \int_0^t(\gamma-\gamma_\varepsilon)(B_{0,t}^1(s) -B^2_{0,t}(s))ds \right]^d,
  \end{eqnarray*}
  for some constant $C$ depending only on $\kappa$. Therefore, the claim (\ref{ecua12}) is reduced to
\[
 \lim_{\varepsilon \rightarrow 0}  \limsup_{t\rightarrow \infty}  \frac 1t	  \log \EE \exp \left \{ C    \int_0^t (\gamma-\gamma_\varepsilon)(B_{0,t}^1(s)-B^2_{0,t} (s))  ds \right\}  \le 0,
\]
 which follows from Lemma 5.3 in \cite{HLN}.   This completes the proof of (\ref{ecua11}).

 To show \eqref{ecua22}, we use the estimate
 \begin{align*}
 	Q_t(\gamma_0- \gamma_{0,\delta},\gamma_\varepsilon)
 	\le\frac{n(n-1)t^2}2\|\gamma_\varepsilon\|_{L^\infty(\RR^\ell)} \|\gamma_{0}-\gamma_{0,\delta}\|_{L^1([0,1])}
 \end{align*}
 to obtain
 \begin{equation*}
 	\frac1t\log\EE e^{\frac q{(n-1)t}Q_t(\gamma_0- \gamma_{0,\delta},\gamma_\varepsilon)}\le \frac n2\|\gamma_\varepsilon\|_{L^\infty(\RR^\ell)} \|\gamma_{0}-\gamma_{0,\delta}\|_{L^1([0,1])}\,.
 \end{equation*}
 This implies \eqref{ecua22} since $\gamma_0\in L^1([0,1])$ and $\lim_{\delta\downarrow0}\gamma_{0,\delta}=\gamma_0$ in $L^1([0,1])$.

	\medskip
	\noindent
	{\it Step 2.}  We claim that
	\begin{equation}  \label{w1}
	\limsup_{t\to\infty}\frac1t \log   \EE \exp \left \{  \frac p {(n-1)t} Q_{t}\left(\gamma_{0,\delta}, \gamma_\varepsilon \right)\right\} \le n  \mathcal{E}(\frac p2\gamma_{0,\delta},\gamma_\varepsilon).
	\end{equation}
Notice that the function $\gamma_\varepsilon$ is bounded and can be expressed in the form
 $\gamma_\varepsilon=K_\varepsilon * K_\varepsilon$, where    the function $K_\varepsilon$, defined by
\begin{equation}
K_{\varepsilon}(x)=\frac{1}{(2\pi)^{\ell}} \int_{\RR^{\ell}} e^{i\xi\cdot x-\frac{\varepsilon}{2}|\xi|^2} \sqrt{f(\xi)}d\xi,
\end{equation}	     
 is bounded with bounded first derivatives, symmetric and $K_\varepsilon \in L^2 (\RR^\ell)$.
	Let $\lambda$ be a fixed number in $(0,1)$ and set {$\rho_{\lambda}=\int_{[0,1]^2\setminus[0,\lambda]^2} \gamma_{0,\delta}(s-r)dsdr$}. For each $j\neq k$, we use the estimate
	\begin{align*}
		&\int_0^t\int_0^t \gamma_\varepsilon(B^j_{0,t}(s)-B^k_{0,t}(r))\gamma_{0,\delta}(\frac{s-r}t)dsdr
		\\&\le \int_0^{\lambda t}\int_0^{\lambda t}\gamma_\varepsilon(B^j_{0,t}(s)-B^k_{0,t}(r))\gamma_{0,\delta}(\frac{s-r}t)dsdr
		+\|\gamma_\varepsilon\|_\infty \rho_\lambda t^2
	\end{align*} 
	together with \eqref{id:density} to obtain
	\begin{align*}
		&\EE e^{\frac{p}{(n-1)t}Q_t(\gamma_{0,\delta},\gamma_\varepsilon)}
		\\&\le e^{\frac n2 p\|\gamma_\varepsilon\|_\infty \rho_\lambda t} 
		\EE\exp\left\{\frac{p}{(n-1)t}\sum_{0\le j<k \le n}\int_{[0,\lambda t]^2}\gamma_\varepsilon({B_{0,t}^j(s)-B_{0,t}^k(r)})\gamma_{0,\delta}(\frac{s-r}t)drds \right\}
		\\&\le \frac{e^{\frac n2 p\|\gamma_\varepsilon\|_\infty \rho_\lambda t}}{(1- \lambda)^{\ell/2}} 
		\EE\exp\left\{\frac{p}{(n-1)t}\sum_{0\le j<k \le n}\int_{[0,\lambda t]^2}\gamma_\varepsilon(B^j(s)-B^k(r))\gamma_{0,\delta}(\frac{s-r}t)drds \right\}\,.
	\end{align*}
	At this point, we apply Proposition \ref{prop:gg} to get
	\begin{align*}
		\limsup_{t\to\infty}\frac1t \log  \EE e^{\frac{p}{(n-1)t}Q_t(\gamma_{0,\delta},\gamma_\varepsilon)}
		\le \frac n2 p\|\gamma_\varepsilon\|_\infty \rho_\lambda 
		+\lambda n\cee (\frac {p \lambda}2 \gamma_{0,\delta},\gamma_ \varepsilon).
	\end{align*}
	Passing through the limit $\lambda\uparrow1$, noting that $\rho_\lambda\to0$, we obtain \eqref{w1}.

	\medskip\noindent \textit{Step 3.} We combine \eqref{65}, \eqref{ecua11}, \eqref{ecua22} and \eqref{w1} to get
	\begin{align*}
		\limsup_{t\to\infty}\frac1t\log\EE e^{\frac{1}{(n-1)t}Q_t(\gamma_0,\gamma)}
		\le\limsup_{\varepsilon\downarrow0}\limsup_{\delta\downarrow0} \frac np\cee(\frac p2 \gamma_{0,\delta},\gamma_\varepsilon)
	\end{align*}
	Note that the order of the limits $\delta\downarrow0$ and $\varepsilon\downarrow0$ can not be interchanged. It is evident to check that $\gamma_{0,\delta}\le \gamma_0$ and $\gamma_\varepsilon\le \gamma$ in quadratic sense. Hence, using \eqref{cee.compare} we have
	\begin{equation*}
		\limsup_{t\to\infty}\frac1t\log\EE e^{\frac{1}{(n-1)t}Q_t(\gamma_0,\gamma)}
		\le\frac np\cee(\frac p2 \gamma_{0},\gamma)\,.
	\end{equation*}
	Finally, letting $p\downarrow 1$, we obtain \eqref{tmp.eqt}, which completes the proof.
	\end{proof}
\begin{remark}
	The case of time-independent noises corresponds to  $\alpha_0=0$. In this case, the function $\gamma_0\equiv 1$ can not be written as a convolution of a function with itself. Thus the proof of Proposition \ref{prop:gg} does not work in this case. 
\end{remark}  

\begin{corollary}\label{cor:u}
Suppose that  the covariance of the noise  satisfies \ref{H1} or \ref{H2},  condition \ref{con:S} holds and the spectral measure $\mu$ is absolutely continuous.
Let $u(t,x)$ be the solution to \eqref{SHE} with nonnegative initial condition $u_0$ satisfying condition \eqref{con:u0}. Then for any integer $n \geq 2$, 
\begin{equation}\label{upper bd}
\limsup_{t \to \infty} t^{-\frac{4-\alpha-2\alpha_0}{2-\alpha}} \log \sup_{(x^1, \dots, x^n) \in (\RR^\ell)^n} \frac{\EE \left[ \prod_{j=1}^n u(t,x^j)\right]}{ \prod_{j=1}^n p_t* u_0(x^j)} \leq n \left(\frac{n-1}{2}\right)^{\frac{2}{2-\alpha}}\mathcal{E}(\alpha_0,\gamma)\,.
\end{equation}
\end{corollary}

\begin{proof}
Let $\{B_{0,t}^j(s), s \in [0,t], j=1,\dots, n\}$, be $\ell$-dimensional Brownian bridges from zero to zero. Using the moment formula for the solution (\ref{eqn:FKbridge})  proved in Proposition \ref{prop1}, we have
\begin{eqnarray*}
\EE \left[ \prod_{j=1}^n u(t,x^j) \right]&=&\EE \bigg(\int_{(\RR^{\ell})^n}\prod_{j=1}^n u_0(x^j + y^j) p_t(y^j)\\
&&\times \exp \left\{\frac 12 \sum_{j\neq k}^n \int_{[0,t]^2}\frac{\gamma(B_{0,t}^j(s)-B_{0,t}^k(r)+\frac{s}{t}y^j-\frac{r}{t}y^k +x^j -x^k)}{|s-r|^{\alpha_0}} dr ds \right\} dy\bigg)\\
&\leq&\prod_{j=1}^n  p_t * u_0(x^j)\EE  \exp \left\{\frac 12 \sum_{j\neq k}^n \int_0^t \int_0^t \frac{\gamma(B_{0,t}^j(s)-B_{0,t}^k(r))}{|s-r|^{\alpha_0}} dr ds \right\}\,,
\end{eqnarray*}
where the last inequality follows from  (\ref{est:GGy}).
Then, the upper bound is a consequence of Theorem \ref{thm:expBB}.
\end{proof}

\begin{remark} Using the approach developed in  \cite{Chen}, we can also show the corresponding lower bound in (\ref{upper bd 2}), assuming  \ref{H1} or \ref{H2} and \ref{con:S}  (but not necessarily the absolute continuity of $\mu$).
However, a lower bound similar to that proved in Corollary  \ref{cor:u} cannot be obtained. For this reason, the proof of a lower bound for the exponential growth indices needs a direct approach as it is done in the next section. 

\end{remark}

\section{Exponential growth indices}\label{sec: Exp Grow indices}

In this section we 
denote by $u(t,x)$ the solution to \eqref{SHE} with nonnegative initial condition $u_0$ satisfying condition \eqref{con:u0}. 
The exponential growth indices are  defined as follows:
	\begin{equation}  \label{expg1}
		\lambda_*(n)=\sup\left\{\lambda>0:\liminf_{t\to\infty} t^{-a} \sup_{|x|\ge \lambda t^{ \frac {a+1}2} }\log\EE u^n(t,x)>0\right \}\,
\end{equation}
and
\begin{equation} \label{expg2}
		\lambda^*(n)=\inf\left\{\lambda>0:\limsup_{t\to\infty} t^{-a}  \sup_{|x|\ge \lambda t^{ \frac {a+1}2} }\log\EE  u^n(t,x)<0\right \}\,,
\end{equation}
where we recall that $a=\frac {4-\alpha -2\alpha_0}{2-\alpha}$.

\subsection{Upper bound for $\lambda^*(n)$}
Set
\begin{equation} \label{b}
b=\frac{2a}{a+1}=\frac{4- \alpha-2 \alpha_0}{3- \alpha- \alpha_0}.
\end{equation}
It may be helpful to note that $a,b\in(1,2)$. For each positive number $\beta$, we define two auxiliary functions $\psi_ \beta$ and $\phi_ \beta$. The function $\psi_ \beta:(0,\infty)\to(0,\infty)$ is defined by
\begin{equation}
	\psi_ \beta(w)=\frac12 \beta^2 b^2 w^{2b-2}+ \beta w^{b}
\end{equation}
and $\phi_ \beta:(0,\infty)\to(0,\infty)$ is uniquely defined by the relation
\begin{equation}\label{eqn:phi}
	\beta b(\phi_ \beta(x))^{b-1}=x- \phi_ \beta(x)\,,\quad\forall x>0\,.
\end{equation}
For every fixed $x>0$, $\phi_ \beta(x)$ can be recognized as the unique minimizer of the function 
\begin{equation}\label{eqn:f}
 	y\mapsto f_{\beta,x}(y):=\frac12(y-x)^2+ \beta y^b
\end{equation} 
on $(0,\infty)$. Together with \eqref{eqn:phi}, it follows that
\begin{equation}\label{min.f}
 	f_{\beta,x}(y)\ge \psi_ \beta(\phi_ \beta(x))
\end{equation}
for every $\beta,x>0$. Relation \eqref{eqn:phi} implies that $\phi_ \beta$ is strictly increasing. 

\begin{theorem}   \label{teo1}
	Assume conditions \ref{H1} or \ref{H2}, condition \ref{con:S}  and the absolute continuity of $\mu$.  
Suppose that   $u_0$ satisfies
	\begin{equation}
	 	\int_{\RR^\ell} e^{\beta|y|^b}u_0(y)dy<\infty
	\end{equation} 
	for some $\beta>0$, then
	\[
\lambda^*(n) \le    g_\beta^{-1} \left(   \left(  \frac {n-1} 2 \right)^{ \frac 2 {2-\alpha}}   \mathcal{E}(\alpha_0,\gamma)\right),
\]
where the function  $g_\beta(\lambda)=\psi_ \beta(\phi_ \beta(\lambda))$ is  given by
\begin{equation} \label{gbeta}
g_\beta(\lambda)= \frac 12 \beta^2 b^2 \phi_\beta(\lambda)^{2b-2} + \beta \phi_\beta(\lambda)^b,
\end{equation}
and $\phi_\beta$ is
characterized by \eqref{eqn:phi}.
	\end{theorem}

\begin{proof}
	It suffices to show that for any $\lambda>0$,
	\begin{align*}
		\limsup_{t\to\infty}t^{-a}\log\sup_{|x|\ge \lambda t^{\frac{a+1}{2}}}\EE u^n(t,x)\le n\left(\frac{n-1}2\right)^{\frac2{2- \alpha}}\cee(\alpha_0,\gamma)-n \psi_ \beta(\phi_ \beta(\lambda)).
	\end{align*}	
	 We write
	\begin{equation*}
		\sup_{|x|\ge \lambda t^{\frac{a+1}{2}}}\EE u^n(t,x) \le \left(\sup_{y\in\RR^\ell}\EE \left(\frac{u(t,y)}{p_t*u_0(y)}\right)^n \right)\left(\sup_{|x|\ge \lambda t^{\frac{a+1}{2}}}p_t*u_0(x) \right)^n\,.
	\end{equation*}
	Together with Corollary \ref{cor:u}, it suffices to show the inequality
	\begin{equation}\label{ptu0 upper}
		\limsup_{t\to\infty}t^{-a} \log \sup_{|x|\ge \lambda t^{\frac{a+1}2}}p_t*u_0(x) \le -\psi_ \beta(\phi_ \beta(\lambda)) \,.
	\end{equation}
	We observe that by the  triangle inequality,
	\begin{align*}
		\frac1{2t}|y-x|^2+\beta|y|^b\ge t^a f_{\beta,|x|t^{-\frac{a+1}2}}(|y|t^{-\frac{a+1}2})\,.
	\end{align*}
	Hence, together with \eqref{min.f}, we see that for every $|x|\ge \lambda t^{\frac{a+1}2} $
	\begin{equation*}
		\frac1{2t}|y-x|^2+\beta|y|^b\ge t^a \psi_ \beta (\phi_ \beta(\lambda))\,.
	\end{equation*}
	Thus,
	\begin{equation*}
		\sup_{|x|\ge \lambda t^{\frac{a+1}2}}\int_{\RR^\ell}e^{-\frac1{2t}|y-x|^2}u_0(y)dy\le e^{-t^a \psi_ \beta(\phi_ \beta(\lambda))}\int_{\RR^\ell}e^{\beta|y|^b}u_0(y)dy\,.
	\end{equation*}
	which implies \eqref{ptu0 upper}.
\end{proof}

As $\beta$ tends to infinity, $\phi_\beta(\lambda)$ tends to zero and it behaves as $\left( \lambda  / b\beta \right)^{\frac 1{b-1}}$.
Therefore, $g_\beta(\lambda)$ behaves as $\frac 12 \lambda^2$. These facts lead to the following result.

\begin{corollary} \label{cor2} Under the assumptions of Theorem  \ref{teo1}, if  $u_0$ satisfies $	 	\int_{\RR^\ell} e^{\beta|y|^b}u_0(y)dy<\infty$
for all $\beta>0$, then
	  \[
   \lambda ^*(n) \le
  \sqrt{ 2  \left(  \frac {n-1} 2 \right)^{ \frac 2 {2-\alpha}}   \mathcal{E}(\alpha_0,\gamma) }.
	 \]
\end{corollary}

\subsection{Lower bound for    $ \lambda _*(n) $}
The main result is the following.

\begin{theorem} \label{t1}
Assume conditions \ref{H2} and \ref{con:S}. Suppose that $u_0$ is non-trivial and non-negative. In addition, we assume that
\begin{equation*}
	\limsup_{t\to\infty}\frac1 {t^a}\lt|\sup_{|x|\ge \lambda t^{\frac{a+1}2}} \log p_t*u_0(x) \rt|<\infty 
\end{equation*} 
for any $\lambda>0$. Then,
	  \[
   \lambda _*(n) \ge    a^{\frac a2} (a+1) ^{-\frac {a+1}2}
  \sqrt{ 2  \left(  \frac {n-1} 2 \right)^{ \frac 2 {2-\alpha}}   \mathcal{E}(\alpha_0,\gamma) }.
	 \]
\end{theorem}

\begin{proof}
Set
\[
I(t):= \frac 1 {t^a} \sup_{|x| \ge \lambda t^{\frac { a+1}2}}  \log \EE u^n (t,x) .
\]
To derive a lower bound for $I_t$ we proceed as follows.     We will make use of the notation
	\[
Q_t\gamma(y):= \sum_{0\le j<k \le n}  \int_{[0,t]^2}  \gamma(B^j_{0,t}(s)- B^k_{0,t}(r) + \frac s t  y^j -\frac  rt  y^k)  {|s-r|^{-\alpha_0}} drds.
	\]
Then, by the Feynman-Kac formula for the moments of the solution in terms of Brownian bridges proved in Proposition \ref{prop1}, we have
		\begin{equation} \label{eq2}
\EE u^n(t,x) =\int_{(\RR^\ell)^n}\EE  \prod_{j=1}^n u_0(x+y^j) p_t(y^j) \exp \left\{ Q_t\gamma(y) \right\} dy.
\end{equation}
For each $\ep>0$, $p>1$, applying H\"older inequality,   we see that
		\begin{align}
			\EE u^n(t,x)
			&\ge  \left(\int_{(\RR^\ell)^n}\EE \prod_{j=1}^n u_0(x+y^j) p_t(y^j) \exp \left\{  \frac1p Q_t\gamma_\ep(y) \right\}dy \right)^p\notag
			\\&\quad\times\left(\int_{(\RR^{\ell})^n}\prod_{j=1}^n u_0(x+y^j) p_t(y^j) \EE  \exp \left\{ \frac 1{p-1}   Q_t(
			\gamma_\ep- \gamma)(y) \right\} dy\right)^{1-p}. \label{eq1'}
		\end{align}
		Notice that, from (\ref{est:GGy}) we can write 
		\begin{equation} \label{eq1}
		\EE  \exp \left\{ \frac {1}{p-1}   Q_t(
			\gamma_\ep- \gamma)(y) \right\} \le  \EE  \exp \left\{ \frac {1}{p-1}   Q_t(
			\gamma- \gamma_\ep)(0) \right\}.
			\end{equation}
{Substituting (\ref{eq1}) into  (\ref{eq1'})} yields
		\begin{eqnarray}
			I(t)  &\ge & - \frac {p-1} {t^a} \log   \EE  \exp \left\{ \frac {1}{p-1}   Q_t(
			\gamma- \gamma_\ep)(0) \right\} -n\frac{p-1}{t^a}\lt|\sup_{|x|\ge \lambda t^\frac{a+1}2} \log (p_t*u_0(x))\rt| \nonumber\\
			&&+ \frac p {t^a} \sup_{|x|\ge \lambda t^\frac{a+1}2}  \log 
			\int_{(\RR^\ell)^n}  \prod_{j=1}^n u_0(x+y^j) p_t(y^j) 	\EE  \exp \left\{ \frac 1pQ_t\gamma_\ep(y) \right\} dy\nonumber\\
			&:=& I_1(t)+I_2(t) + I_3(t).\label{eqI12}
		\end{eqnarray}
		Choosing $\varepsilon=\varepsilon(t)=\delta t^{1-a}$ with $\delta>0$, from (\ref{ecua11})  we obtain
			\begin{equation*} 
				\lim_{\delta \rightarrow 0} 	\limsup _{t\rightarrow \infty} \frac 1{t^a} \log   \EE  \exp \left\{ \frac qp   Q_t(\gamma- \gamma_{\delta t^{1-a}})(0) \right\} \le 0.
			\end{equation*}
		In addition, from our assumption,
		\begin{equation*}
			\lim_{p\to1}\limsup_{t\to\infty}I_2(t)=0\,.
		\end{equation*}
		In other words, $I_1(t)$ and $I_2(t)$ are negligible in the limits $t\to\infty$, $\delta\to0$ and $p\to1$. 

		We now consider $I_3$. It can be written as
			\[
			I_3(t)= \frac p {t^a} \sup_{|x|\ge \lambda t^{\frac{a+1}2}}  \log \EE (u_{\ep,p}^n(t,x)),
			\]
			where $u_{\ep,p}(t,x)$ denotes the solution of equation  (\ref{SHE}) with initial condition $u_0$ and spatial covariance {$\frac{1}{p}\gamma_{\ep(t)}$}, where
			$\ep=\ep(t) = \delta t^{1-a}$.  { Define $\mathcal{H}_{p,\varepsilon}$ as in (\ref{innprod1}), but with $\mu(d\xi)$ replaced by $\frac{1}{p}e^{-\varepsilon|\xi|^2}\mu(d\xi)$}. 
			
			For every $\phi$ in {$\mathcal{H}_{p,\varepsilon}$}, we denote by $Z(\phi)$ the (Wick) exponential functional
	\begin{equation*}
		Z(\phi)=\exp\left\{W(\phi)-\frac12\|\phi\|_{\mathcal{H}_{p,\varepsilon}}^2 \rt\}\,.
	\end{equation*}
			By the Feynman-Kac formula for the solution of equation (\ref{SHE}), when the spatial covariance is bounded, we obtain
			\[
			u_{\ep,p}(t,x) =\EE_B \int_{\RR^\ell}  u_0 (x+y) p_t(y) Z(\psi_{x,y})  dy,
			\]
			where  $\psi_{x,y}(s,z)=    \delta(B_{0,t} (t-s) +x + \frac {t-s}t y -z){\bf 1}_{[0,t]}(s)  $ and			\[
			\| \psi_{x,y}\|^2_{{\mathcal{H}_{p,\varepsilon}}} =\int_{[0,t]^2} { \frac{1}{p}} \gamma_\ep\left(B_{0,t}(s) -B_{0,t}(r) +\frac {s-r}t y\right) |s-r|^{-\alpha_0} dsdr .
			\]
Let $q$ be such that $\frac1n+\frac1q=1$. Using H\"older inequality, for any $\phi \in \mathcal{H}_{p,\varepsilon}$ we have
	\begin{eqnarray}  \nonumber
	\EE(u_{\ep,p}^n(t,x))  &=&\EE_W \left( \EE_B \int_{\RR^\ell}  u_0(x+y) p_t(y) Z(\psi_{x,y})  dy \right)^n\\ \nonumber
	&\ge &  \|Z(\phi)\|_{L^q(\Omega)}^{-n}  \left( \EE_W  \left(Z(\phi) \EE_{B} \int_{\RR^\ell}  u_0(x+y) p_t(y) Z(\psi_{x,y})  dy \right) \right)^n\\ \nonumber
	&=& \exp\lt\{-\frac{n}{2(n-1)}\|\phi\|^2_{\mathcal{H}_{p,\varepsilon}} \rt\}  \\
	&&\times \left( \int_{\RR^\ell}  u_0(x+y) p_t(y)    \EE_B [\exp\{ \langle \phi, \psi_{x,y} \rangle_{\mathcal{H}_{p,\varepsilon}} \} ] dy  \right)^n. \label{eq5'}
	\end{eqnarray}
	We are going to choose an element $\phi$, which depends on $t$ and $x$.
	
	Our next step is the computation of the inner product  $ \langle \phi, \psi_{x,y} \rangle_{\mathcal{H}_{p,\varepsilon}} $. We can write
	\begin{eqnarray*}
	\langle \phi, \psi_{x,y} \rangle_{\mathcal{H}_{p,\varepsilon}} &=& \frac{1}{p}\int_0^t \int_0^t |s-r|^{-\alpha_0}   \int_{\RR^\ell} \phi(r,z) \gamma_\ep(B_{0,t}(t-s) +x +\frac {t-s}t y-z)dzdsdr\\
	&=& \frac{1}{p} \int_0^t \int_0^t |s-r|^{-\alpha_0}   \int_{\RR^\ell} \phi(t-r,z+x) \gamma_\ep(B_{0,t}(s) +\frac {s}t y-z)dzdsdr.
	\end{eqnarray*}
 Set
	\begin{equation*}
	t_n=c t^a,  
	\end{equation*}
	where $a=\frac{4-\alpha-2\alpha_0}{2-\alpha}$ and  $c=(n-1)^{\frac{2}{2-\alpha}}$.
	Making   the change of variables $s \rightarrow \frac{t}{t_n} s$ and $r \rightarrow tr$ and using the scaling property for Brownian bridge, we obtain that
	 \begin{eqnarray*}
	 \langle \phi, \psi_{x,y} \rangle_{\mathcal{H}_{p,\varepsilon}} &=& \frac{1}{p}  c^{\frac \alpha 2-1}  \int_0^1 \int_0^{t_n}  |\frac{s}{t_n}-r|^{-\alpha_0}   \int_{\RR^\ell} \phi(t-rt,z+x) \\
	 &&\times \gamma_{\ep \frac {t_n}t}\left(B_{0,t_n}(s) +\frac {s}{\sqrt{t_nt}}    y- \sqrt{\frac {t_n} t}z\right)dzdsdr.
	 \end{eqnarray*}
	 Finally, the change of variables $z\rightarrow  \sqrt{\frac {t} {t_n}}z$ yields
  \begin{eqnarray*}
	 \langle \phi, \psi_{x,y} \rangle_{\mathcal{H}_{p,\varepsilon}} &=& \frac{1}{p} \left( \frac  t {t_n}  \right)^{\frac \ell 2} c^{\frac \alpha 2-1}  \int_0^1 \int_0^{t_n}  |\frac{s}{t_n}-r|^{-\alpha_0}   \int_{\RR^\ell} \phi(t-rt, \sqrt{\frac  t {t_n}} z+x) \\
	 &&\times \gamma_{ c\delta }\left(B_{0,t_n}(s) +\frac {s}{\sqrt{t_nt}}    y- z\right)dzdsdr.
	 \end{eqnarray*}
Choosing $\phi$ of the form
	\[
	\phi(r,z)=  \left( \frac  {t_n} t \right)^{\frac \ell 2} c^{1-\frac \alpha 2}  \widehat{\phi} \left( \frac {t-r}t, \sqrt{\frac {t_n}t}(z - x)\right) \mathbf{1}_{[0,t]}(r),
	\]
	where $\widehat{\phi}$ satisfies
	\begin{equation}   \label{h1}
		\sup_{r\in [0,1]} \int_{\RR^\ell} |\widehat{\phi}(r,y)|dy <\infty,
		\end{equation} 
	we can write
	\[
		\langle \phi, \psi_{x,y}\rangle_{\mathcal{H}_{p,\varepsilon}}
		= \frac{1}{p} \int_0^1 \int_0^{t_n}  |\frac{s}{t_n}-r|^{-\alpha_0}   \int_{\RR^\ell}  \widehat{\phi}(r, z)   \gamma_{ c\delta }\left(B_{0,t_n}(s) +\frac {s}{\sqrt{t_nt}}    y- z\right)dzdsdr.
		\]
		Set
		\[
		f(s, w) = \frac{1}{p}\int_0^1 \int_{\RR^\ell}   \frac {\widehat{\phi}(r,z) \gamma_{c\delta} (w-z)}{| s -r |^{\alpha_0}}  dzdr.
		\]
		Then, we obtain
		\begin{equation}  \label{eq6}
		\langle \phi, \psi_{x,y}\rangle_{\mathcal{H}_{p,\varepsilon}}=\int_0^{t_n} f(\frac s{t_n}, B_{0,t_n}(s) +\frac {s}{\sqrt{t_nt}}    y)ds.
		\end{equation}
		On the other hand, for this choice of $\phi$, we obtain
	 \begin{eqnarray*}
		\| \phi\|_{\mathcal{H}_{p,\varepsilon}} ^2&=& \frac{1}{p}\left( \frac {t_n} t \right)^{\ell} c^{2-\alpha}   \int_0^t \int_0^t |s-r|^{-\alpha_0} \\
		&&\times \int_{(\RR^\ell)^2} 
		 \widehat{\phi} \left( \frac {t-r}t, \sqrt{\frac {t_n}t} (z- x)\right) \widehat{\phi} \left( \frac {t-s}t, \sqrt{\frac {t_n}t}{w - x}\right) \gamma_{\ep} (z-w) dzdwdrds.
		 \end{eqnarray*}
		The change of variables  $s\rightarrow  t-ts$, $r\rightarrow t-tr$, $z\rightarrow  \sqrt{\frac  t{t_n }} z + x$ and  $w\rightarrow  \sqrt{\frac  t{t_n }} w + x$ leads to
			\begin{equation}  \label{eq7}
		\| \phi\|_{\mathcal{H}_{p,\varepsilon}} ^2 = \frac{1}{p} t^a c^{2-\frac \alpha 2}   \int_0^1 \int_0^1  \int_{(\RR^\ell)^2} 
		\frac{   \widehat{\phi} (r,z) \widehat{\phi} (s,w)}{ |s-r|^{-\alpha_0} } \gamma_{c\delta} (z-w) dzdwdrds.
		 \end{equation}
		Substituting (\ref{eq6}) and (\ref{eq7}) into (\ref{eq5'}), we get
			 \begin{eqnarray*}
		\frac p {t^a} \log \EE(u_{\ep,p}^n(t,x)) &\geq&-\frac {n}{2(n-1)} c^{2-\frac \alpha 2}   \int_0^1 \int_0^1 \int_{(\RR^\ell)^2} 
		\frac { \widehat{\phi} (r,z) \widehat{\phi} (s,w)}{|s-r|^{\alpha_0}} \gamma_{c\delta} (z-w) dzdwdrds \\
		&&+\frac {np} {t^a}  \log \int_{\RR^\ell}  u_0 (x+y) p_t(y)    \EE_B \exp\left\{ \int_0^{t_n} f(\frac s{t_n}, B_{0,t_n}(s) +\frac {s}{\sqrt{t_nt}}    y)ds   \right\}  dy.
			 \end{eqnarray*}
		This together with \eqref{eqI12} leads to the inequality, for any $K>0$,
		\[
		I_3(t) \ge I_{3,1} + I_{3,2}(t) + I_{3,3}(t) ,
		\]
		where
		\[
		I_{3,1}=-\frac {n}{2(n-1)} c^{2-\frac \alpha 2}   \int_0^1 \int_0^1 \int_{(\RR^\ell)^2} 
		\frac { \widehat{\phi} (r,z) \widehat{\phi} (s,w)}{|s-r|^{\alpha_0}} \gamma_{c\delta} (z-w) dzdwdrds,
		\]
		\[
		I_{3,2}({t,x}) =  \frac {np} {t^a}   \log \int_{|y| \le K\sqrt{tt_n}}  u_0(x+y) p_t(y)dy
		\]
		and
		\[
		I_{3,3}(t)=\frac {np} {t^a}   \inf_{|y| \le Kt_n} \log \EE_B \exp\left\{ \int_0^{t_n} f(\frac s{t_n}, B_{0,t_n}(s) +\frac {s}{t_n}    y)ds\right\} .
		\]
We are going to analyze these three terms and this will be done in several steps.

\medskip
\noindent
{\it Step 1}. Using the properties of the initial condition, we claim that if $\lambda < K\sqrt{c}$, then
	\begin{equation}\label{eq11}
		\liminf_{t\rightarrow \infty} I_{3,2}(t) \ge -\frac {np}2 K^2c.
	\end{equation}
	Notice first that  $\sqrt{tt_n}= \sqrt{c} t^{\frac {a+1} 2}$. 	Recall  that $u_0$ is non-trivial, there exists $M>0$ such that $\int_{|y|\le M}u_0(y)dy>0$. For $t$ large enough, 
	$ \lambda t^{\frac { a+1}2} +M \le  K\sqrt{c} t^{\frac {a+1} 2}$.  Therefore,  choosing $x_0$ such that $|x_0| =\lambda t^{\frac { a+1}2}$ implies that 
	\[
	\{y: |x_0+ y|\le M \} \subset \{y:  |y|\le K\sqrt{c} t^{\frac {a+1} 2} \}.
	\]
Thus we obtain 
		\begin{equation*}  
	\liminf_{t\rightarrow \infty} I_{3,2}(t) \ge  \liminf_{t\rightarrow \infty}  \frac {np} {t^a} \log\int_{|x_0 +y| \le M} e^{-\frac{K^2 c t^a}2} u_0(x_0+y) dy=-\frac {np}2 K ^2c,
		\end{equation*}
	which is \eqref{eq11}.
	
\medskip
\noindent
{\it Step 2}.  We can write 
	\[
	\liminf_{t\rightarrow \infty} I_{3,3}(t) =\liminf_{t\rightarrow \infty} \frac {npc} {t}   \inf_{|y| \le Kt } \log \EE_B \exp\left\{ \int_0^{t} f(\frac s{t}, B_{0,t}(s) +\frac {s}{t}    y)ds\right\}  
	\]
	For any $\rho \in (0,1)$, we can write
	\[
	\EE_B \exp\left\{ \int_0^{t} f(\frac s{t}, B_{0,t}(s) +\frac {s}{t}    y)ds\right\}  
\ge \EE_B  \exp\left\{ \int_0^{\rho t} f(\frac s{t}, B_{0,t}(s) +\frac {s}{t}    y)ds\right\}.
\]
From (\ref{id:density}), we get
\begin{multline} 
 \EE_B \exp\left\{ \int_0^{\rho t} f(\frac s{t}, B_{0,t}(s) +\frac {s}{t}    y)ds\right\} 
\\\ge(1-\rho)^{-\frac \ell 2}
   \EE_B \left({\bf 1}_{ A_R}   \exp\left\{ \int_0^{\rho t} f(\frac s{t}, B(s))ds
+\frac {|y|^2} {2t} -\frac {|y-B(\rho t)|^2} {2t(1-\rho)} \right\} \right),
	\end{multline}
where $A_R =\{ \sup_{0\le s\le \rho t} |B(s)| \le R\}$ for $R >0$. Notice that, if $|y| \le Kt$, on the set $A_R$ we have
	\begin{equation} \label{eq9}
\frac {|y|^2} {2t} -\frac {|y-B(\rho t)|^2} {2t(1-\rho)} 
\ge - \frac \rho {1-\rho} \frac {K^2}2 t -\frac {KR} {(1-\rho)} -\frac {R^2} {2t(1-\rho)}.
	\end{equation}
On the other hand, by Proposition 3.1 of  \cite{CHSX} we obtain
\[
\lim_{t\rightarrow \infty} \frac 1t \log  \EE_B \left(1_{ A_R}   \exp\left\{ \int_0^{\rho t} f(\frac s{t}, B(s))ds \right\} \right)
=\rho  \int_0^1 \Lambda_R(f(\rho s, \cdot))ds,
\]
	where
	\[
	 \Lambda_R(f(\rho s, \cdot)) =\sup_{g\in \mathcal{F}_\ell (B_R)} \left\{ \int_{B_R} f(\rho s,x) g^2(x) dx -\frac 12 \int_{B_R} | \nabla g(x)|^2dx\right\},
	 \]
	 and $\mathcal{F}_\ell (B_R)$ is the set  of smooth functions on $B_R:=\{ x: |x| \le R\}$ with $\|g\|_{L^2(B_R)}=1$ and $g(\partial B_R)=\{0\}$.
For this result we need that for each $0\le s\le 1$, the function $f(\rho s, \cdot)$ is bounded and continuous and the family of  functions
 $\{s\rightarrow f(\rho s, x), x\in \RR^\ell\}$ is  equicontinuous in $[0,1]$. These properties are a consequence of assumption (\ref{h1}).
	In conclusion, we have proved that
	\begin{equation} \label{eq8}
	\liminf_{t\rightarrow \infty} I_{3,3}(t) \ge -npc \frac \rho {1-\rho} \frac {K^2}2+cnp \rho   \int_0^1 \Lambda_R(f(\rho s, \cdot))ds.
	\end{equation}
	From  (\ref{eq8}), (\ref{eq9}) and (\ref{eq11}),  letting 
	$K\downarrow  \lambda /\sqrt{c}$ and $R\uparrow \infty$, we obtain
	\begin{multline}   \label{eq12}  
		\liminf_{t\rightarrow \infty} (I_{3,2}(t)+ I_{3,3}(t) ) \ge  -\frac {np}{2(1-\rho)}  \lambda^2 \\ 
		 + nc \rho p
		\left(\int_0^1 \int_{\RR^\ell} f(s\rho, x) g^2(s,x) dxds - \frac 12 \int_0^1 \int_{\RR^\ell} |\nabla g(s,x)|^2 dxds \right),
		\end{multline}
	for any function $g(s,x)$ in $\mathcal{A}_\ell$, where $\mathcal{A}_\ell$ has been defined in (\ref{aell}).
	We can write
	\[
	\int_0^1 \int_{\RR^\ell} f(s\rho, x) g^2(s,x) dsdx
	=\frac{1}{p}\int_0^1 \int_0^1   \int_{\RR^{2\ell}}  \frac{ \widehat{\phi} (r,y) g^2(s,x) }{|\rho s -r|^{\alpha_0} } \gamma_{c\delta} (x-y)  dx dy dsdr.
	\]
	 Making the change of variables $s\rho \rightarrow s$, yields
	 	\[
	\int_0^1 \int_{\RR^\ell} f(s\rho, x) g^2(s,x) dsdx
=\frac 1{p\rho}\int_0^1 \int_0^\rho   \int_{\RR^{2\ell}}  \frac{\widehat{\phi} (r,y) g^2( s / \rho ,x)  } {|s -r|^{\alpha_0} }\gamma_{c\delta} (x-y) dx dy dsdr.
	\]
	 Now choose the function $\widehat{\phi}$ of the form $\widehat{\phi} (r,x)= g^2(\frac r\rho, x) {\bf 1}_{[0,\rho]} (r)$. With this choice we obtain
	 \begin{eqnarray*} 
	\int_0^1 \int_{\RR^\ell} f(s\rho, x) g^2(s,x) dsdx
&{\geq}&\frac {1}{p\rho}\int_0^\rho \int_0^\rho   \int_{\RR^{2\ell}}  \frac{ g^2 (r / \rho, y) g^2( s/ \rho ,x) }{|s -r|^{\alpha_0} } \gamma_{c\delta} (x-y)  dx dy dsdr \\
&=&\frac{1}{p}\rho ^{1-\alpha_0} \int_0^1 \int_0^1   \int_{\RR^{2\ell}}  \frac{ g^2 (r, y) g^2( s ,x) }{|s -r|^{\alpha_0} } \gamma_{c\delta} (x-y)  dx dy dsdr.
	 \end{eqnarray*}
	 \medskip
	 \noindent
	 {\it Step 3}. 	  With the above choice for $\widehat{\phi}$ and letting $p\rightarrow 1$, the term $I_{3,1}$ can be written as
		 \begin{eqnarray}  \nonumber 
		 I_{3,1}  &=&  -\frac n{2(n-1)}  c^{2-\frac \alpha 2}   \int_0^\rho \int_0^\rho \int_{\RR^{2\ell}} 
		   \frac{g^2 ( r/ \rho, y) g^2(s/  \rho ,x) } { |s-r|^{\alpha_0} }  \gamma_{c\delta} (z-x) dxdydrds\\ \label{eq13}
		   &=&-\frac {nc}{2}     \rho^{2-\alpha_0}    \int_{[0,1]^2}  \int_{\RR^{2\ell}} 
		   \frac{g^2 ( r, y) g^2(s ,x)} { |s-r|^{\alpha_0} }  \gamma_{c\delta} (z-x) dxdydrds.
	 \end{eqnarray}
	Finally, from (\ref{eq12}) and (\ref{eq13}), we obtain
	\begin{multline*}
	 	\lim_{p \downarrow 1}\liminf_{t\rightarrow \infty}  I_3(t)    \ge  -\frac n{2(1-\rho)}  \lambda^2 + \\
		nc \rho	\left( \frac{\rho^{1-\alpha_0}} 2\int_{[0,1]^2} \int_{\RR^{2\ell}} \frac{ g^2 (r, y) g^2( s ,x) }{|s -r|^{\alpha_0} }\gamma_{c\delta} (x-y)dxdydrds - \frac 1 2 \int_0^1 \int_{\RR^\ell} |\nabla g(s,x)|^2 dxds \right),
	\end{multline*}
Letting $\delta \downarrow 0$, we obtain
	\begin{multline*}
	 	\lim_{\delta\downarrow0,p \downarrow 1}\liminf_{t\rightarrow \infty}  I_3(t)    \ge  -\frac n{2(1-\rho)}  \lambda^2 + \\
		nc \rho	\left( \frac{\rho^{1-\alpha_0}} 2\int_{[0,1]^2} \int_{\RR^{2\ell}} \frac{ g^2 (r, y) g^2( s ,x) }{|s -r|^{\alpha_0} }\gamma (x-y)dxdydrds - \frac 1 2 \int_0^1 \int_{\RR^\ell} |\nabla g(s,x)|^2 dxds \right),
	\end{multline*}
Now we write $\widehat{g}(r,x)= \sqrt{\varkappa}  g(r,\varkappa x)$ {where $\varkappa$ is a constant whose value will be determined very soon, } and we obtain, using the scaling properties of $\gamma$,
\begin{eqnarray*}
&& \frac{\rho^{1-\alpha_0}} 2 \int_{[0,1]^2}\int_{\RR^{2\ell}}  \frac{ \widehat{ g}^2(r,y) \widehat{g}^2(s,x)}{|s-r|^{\alpha_0}} \gamma(x-y)dxdydrds-\frac 1 2 \int_0^1 \| \nabla \hat{g}(s,\cdot)\|^2_{L^2(\RR^\ell)}ds \\
&&\quad =\frac {\varkappa^\alpha  \rho^{1-\alpha_0} } 2 \int_{[0,1]^2}\int_{\RR^{2\ell}}  \frac{  g^2(r,y) g^2(s,x)}{|s-r|^{\alpha_0}} \gamma(x-y)dxdydrds-\frac {\varkappa^2} 2 \int_0^1 \| \nabla g(s,\cdot)\|^2_{L^2(\RR^\ell)}ds
\end{eqnarray*}	
 Finally, choosing $\varkappa =2^{ \frac 1{\alpha-2}} \rho ^{\frac {1-\alpha_0}{2-\alpha}}$ and taking  the supremum over $g$, we obtain 
 \[
 \liminf_{t\rightarrow \infty}  I_3(t)   \ge  -\frac n{2(1-\rho)}  \lambda^2 
		+ n  \rho^{a}     \left( \frac {n-1} 2 \right) ^{\frac 2 {2-\alpha}} \mathcal{E}(\alpha_0,\gamma).
		\]
Optimizing in $\rho$, this produces the lower bound
\[
\lambda_*(n) \ge  a^{\frac a2} (a+1) ^{-\frac {a+1}2}   \sqrt{2 \left( \frac {n-1} 2 \right) ^{\frac 2 {2-\alpha}} \mathcal{E}(\alpha_0,\gamma)}.
\]
The proof is now complete.
\end{proof}

\begin{remark}
	Putting together the results from Corollary \ref{cor2} and Theorem   \ref{t1} we obtain, for a nontrivial $u_0$ with compact support and assuming a covariance satisfying  conditions \ref{H2}, \ref{con:S} and the absolute continuity of $\mu$,
\[
  a^{\frac a2} (a+1) ^{-\frac {a+1}2}   \sqrt{2 \left( \frac {n-1} 2 \right) ^{\frac 2 {2-\alpha}} \mathcal{E}(\alpha_0,\gamma)}\le \lambda_*(n) \le
  \lambda^*(n) \le  \sqrt{2 \left( \frac {n-1} 2 \right) ^{\frac 2 {2-\alpha}} \mathcal{E}(\alpha_0,\gamma)}.
\]
Notice that when $\alpha_0 \uparrow 1$ the constant $a$ converges to  1 and the above factor converges to $\frac 12$.  In this sense, in comparison with (\ref{id:speed1}), our result is not optimal. We conjecture that the constant in the left-hand side should be $1$, but our techniques do not allow to show this.

\end{remark}

\bigskip
 \noindent
	\textbf{Acknowledgment:} 	J. Huang and K. L\^e were  supported by the NSF  Grant no. 0932078 000, while  they visited the Mathematical Sciences Research Institute in Berkeley, California in Fall 2015, during which the project was initiated.  D. Nualart was supported by the NSF grant no.  DMS1512891 and the ARO grant no. FED0070445. K. L\^e thanks PIMS for its support through the Postdoctoral Training Centre in Stochastics. The authors wish to thank the referees who pointed out a gap in the original proof of Theorem \ref{thm:expBB}, which leads to a new proof in the revised version.

\end{document}